\documentclass[reqno]{amsart}[14]
\usepackage{amssymb, amsmath}
\allowdisplaybreaks %long equations
\usepackage{mathrsfs}
\usepackage{comment}
\usepackage{color}
\usepackage{tikz-cd}

\numberwithin{equation}{section} %enumerate equations with section number

\def\sideremark#1{\ifvmode\leavevmode\fi\vadjust{\vbox to0pt{\vss% the remark
 \hbox to 0pt{\hskip\hsize\hskip1em%                          will appear only
 \vbox{\hsize3cm\tiny\raggedright\pretolerance10000%          on the side
 \noindent #1\hfill}\hss}\vbox to8pt{\vfil}\vss}}}%

\theoremstyle{plain} \newtheorem{thm}{Theorem}[section]
\theoremstyle{plain} 
\theoremstyle{plain} \newtheorem{lem}[thm]{Lemma}
\theoremstyle{plain} \newtheorem{prop}[thm]{Proposition}
\theoremstyle{plain} \newtheorem{cor}[thm]{Corollary}
\theoremstyle{definition} \newtheorem{defn}{Definition}
\theoremstyle{plain} \newtheorem{rema}[thm]{Remark}
\newtheorem{exas}[thm]{Examples}

\newcommand{\be}{\begin{equation}}
\newcommand{\ee}{\end{equation}}
\newcommand{\bea}{\begin{eqnarray}}
\newcommand{\eea}{\end{eqnarray}}
\newcommand{\beas}{\begin{eqnarray*}}
\newcommand{\eeas}{\end{eqnarray*}}

\newcommand{\R}{\mathbb{R}}
\newcommand{\C}{\mathbb{C}}
\newcommand{\N}{\mathbb{N}}
\newcommand{\Z}{\mathbb{Z}}
\newcommand{\Zs}{\Z_2} 		%super
\newcommand{\Zn}{\Zs^n} %multigraded (changed to eliminate parentheses)

\newcommand{\cO}{{\mathcal{O}}} %sheaf
\newcommand{\cJ}{\mathcal{J}}
\newcommand{\cT}{\mathcal{T}}

\newcommand{\cF}{{\mathcal{F}}}

\newcommand{\Ci}{{\mathcal{C}}^{\infty}} %smooth

\newcommand{\op}[1]{\!\!\mathop{\rm ~#1}\nolimits}

\DeclareMathOperator{\id}{id}

\newcommand{\I}{\mathbb{I}}			%id matrix

\newcommand{\ev}{\textrm{ev}}								%evalutaion map
\DeclareMathOperator{\Hom}{Hom}								%categorical hom
\newcommand{\uHom}{\underline{\Hom}} % internal hom (module)
\newcommand{\GL}{\mathsf{GL}^{\mathbf{0}}}		% invertible (0-deg) matrices
\newcommand{\gl}{\mathsf{gl}}									%graded matrices
\newcommand{\ModA}{\Zn\mathtt{Mod}(A)} %A-module--needs revision?
\newcommand{\ots}{\otimes}
\newcommand{\vect}[1]{{\bf #1}}    %vector--needs revision?
\newcommand{\mathsc}[1]{{\mathscr #1}}   %mathsc--needs revision?
\newcommand{\mathcs}[1]{{\mathscr #1}}
\newcommand{\p}{\partial}
\newcommand{\zL}{\Lambda}
\begin{document}
%%%%%%%%%%%%%%%%%%%%%%%%%%%%%%%%%%%%%%%%%%%%%%%%%%%%%%%%%%%%%%%
%%%%%%%%%%%%%%%%%%%%%%%%%%%%%%%%%%%%%%%%%%%%%%%%%%%%%%%%%%%%%%%

%------------------------vertical space text-equation-text
\belowdisplayskip=7pt plus 3pt minus 4pt
\belowdisplayshortskip=7pt plus 3pt minus 4pt
%---------------------------------------------------------

\title{Local Forms of Morphisms of Colored Supermanifolds}
\author{Tiffany Covolo, Stephen Kwok \& Norbert Poncin}
\maketitle

\begin{abstract} In \cite{Covolo:2016}, \cite{Covolo:2012} and \cite{Poncin:2016}, we introduced the category of colored supermanifolds ($\Zn$-super\-ma\-ni\-folds or just $\Zn$-manifolds ($\Zn=\Z_2\times\ldots\times\Z_2$ ($n$ times))), explicitly described the corresponding $\Zn$-Berezinian and gave first insights into $\Zn$-integration theory. The present paper contains a detailed account of parts of the $\Zn$-differential calculus and of the $\Zn$-variants of the trilogy of local theorems, which consists of the inverse function theorem, the implicit function theorem and the constant rank theorem. \end{abstract}

\small{\vspace{2mm} \noindent {\bf MSC 2010}: 17A70, 58A50, 13F25, 16L30 \medskip

}

\tableofcontents
%-%-%-%-%-%-%-%-%-%-%-%-%-%-%

\newcommand{\z}{\Z_2^n}

\section*{Introduction}
Loosely speaking, colored supermanifolds, also referred to as $\Z_{2}^{n}$-supermanifolds or just $\z$-manifolds ($\z=\Z_2^{\times n}$), are `manifolds' for which the structure sheaf has a $\Z_{2}^{n}$-grading and the commutation rules for the local coordinates come from the standard scalar product of $\z$ (see \cite{Bruce:2019b, Bruce:2018, Bruce:2019, LinManLinAct, Covolo:2016,Covolo:2016a, Covolo:2016c, Covolo:2012, Poncin:2016} for details). This is not just a trivial or straightforward generalization of the notion of a standard supermanifold, as one has to deal with formal coordinates that anticommute with other formal coordinates, but are themselves \emph{not} nilpotent. Due to the presence of formal variables that are not nilpotent, formal power series in the formal parameters are used rather than polynomials.\medskip

The motivation to introduce and study $\Z_2^n$-geometry comes from various sources. First, $\Z_{2}^{n}$-gradings ($n \geq 2$) can be found in  the theory of  parastatistics (see for example \cite{Druhl:1970,Green:1953,Greenberg:1965,Yang:2001}) and in relation to an alternative approach to supersymmetry \cite{Tolstoy:2013}. `Higher graded'  generalizations of the super Schr\"{o}dinger algebra (see \cite{Aizawa:2017}) and the super Poincar\'{e} algebra (see \cite{Bruce:2019a}) have  appeared in the literature.  Furthermore, such gradings appear in the theory of mixed symmetry tensors as found in string theory and some formulations of supergravity (see \cite{Bruce:2018a}).   It must also be pointed out that  quaternions and more general Clifford algebras can  be understood as $\Z_2^n$-graded $\Z_2^n$-commutative algebras (see \cite{Albuquerque:1999, Albuquerque:2002, MO1, MO2}). Finally,  \emph{any} `sign rule' can be interpreted in terms of a $\Z_{2}^{n}$-grading (see \cite{Covolo:2016}). \medskip

The theory of $\Z_2^n$-geometry is currently being developed. Although the available results include the $\z$-Berezinian (see \cite{Covolo:2012}) and low dimensional $\z$-integration theory (see \cite{Poncin:2016}), many foundational questions remain. The present paper deals with parts of the $\z$-differential calculus and contains the trilogy of local theorems, which consists of the inverse function theorem, the implicit function theorem and the constant rank theorem. These $\z$-geometric results are formally similar to their super-geometric counterparts, but their proofs are often subtler. On the other hand, integration on $\z$-manifolds turns out to be fundamentally different from integration on supermanifolds. The novel aspect of integration on $\z$-manifolds is integration with respect to the non-zero degree even parameters.\medskip

It is worth noting that the concept of the functor of points is crucial in $\Z_2^n$-geometry. The functor of points has been used informally in Physics as from the very beginning. It is actually of importance in the contexts where there is no good notion of point as in super- and $\Z_2^n$-geometry and in algebraic geometry. For instance, homotopical algebraic geometry (over differential operators) \cite{TVI, TVII, BPP:KTR, BPP:HAC} is completely based on the functor of points approach. In $\Z_2^n$-geometry, we are particularly interested in functors of $\zL$-points, i.e., functors of points from appropriate locally small categories $\tt C$ to a functor category whose source is not the category ${\tt C}^{\op{op}}$ but the category ${\tt G}$ of $\z$-Grassmann algebras $\zL\,$. However, functors of points that are restricted to the very simple test category ${\tt G}$ are fully faithful only if we replace the target category of the functor category by a subcategory of the category of sets. Examples of categories ${\tt C}$ whose functors of $\zL$-points are fully faithful contain the categories of $\Z_2^n$-manifolds, linear $\Z_2^n$-manifolds, $\Z_2^n$-graded vector spaces (zero degree rules functor), the category of $\Z_2^n$-Lie groups... \cite{Bruce:2019b, LinManLinAct}. In the case of $\Z_2^n$-manifolds, for example, the target category of the functor category is the category of Fr\'echet manifolds over commutative Fr\'echet algebras. \medskip

For various sheaf-theoretical notions we will draw upon Hartshorne \cite[Chapter II]{Hartshorne:1977} and Tennison \cite{Tennison:1975}. There are several good introductory books on the theory of supermanifolds including Bartocci, Bruzzo \&  Hern\'{a}ndez-Ruip\'{e}rez \cite{Bartocci:1991}, Bernstein, Leites, Molotkov \& Shander \cite{Bernstein:2013}, Carmeli, Caston \& Fioresi \cite{Carmeli:2011}, Deligne \& Morgan \cite{Deligne:1999}, Leites \cite{Leites:1980} and Varadarajan \cite{Varadarajan:2004}. For categorical notions we will be based on Mac Lane \cite{MacLane1998}.

%-%-%-%-%-%-%-%-%-%-%-%-%-%-%

\section{Preliminaries}\label{sec:prelim}
%-%-%-%-%-%-%-%-%-%-%-%-%-%-%

In this section, we will fix the notation used throughout the article and recall some basic definitions and results. For further details we refer the reader to the first two articles of this series on $\Zn$-graded geometry \cite{Covolo:2016}, \cite{Covolo:2016a}, as well as previous (resp., follow up) papers of the authors on $\Zn$-graded algebra \cite{Covolo:2012},\cite{CM} (resp., $\Zn$-integration \cite{Poncin:2016}), and references therein.

\subsection{$\Zn$-superalgebra}
In the sequel, $\mathbb{K}$ always denotes a field of characteristic $0$. In our notation, $\Zn=\Zs\times\cdots\times\Zs$ ($n$-times).\medskip

A \emph{$\Zn$-graded algebra} $A$ (over $\mathbb{K}$), is a $\mathbb{K}$-algebra of the form $A=\oplus_{\gamma \in\Zn} A^{\gamma}$ (decomposition into $\mathbb{K}$-vector spaces), in which the multiplication respects the $\Zn$-degree, i.e., $A^{\alpha}\cdot A^{\beta} \subset A^{\alpha+\beta}$. In the following, we will always assume algebras to be associative and unital ($1\in A^{0}$).
If in addition, for any pair of homogeneous elements $a\in A^{\alpha}$ and $b\in A^{\beta}$,
\be\label{signrule}
ab=(-1)^{\langle \alpha,\beta \rangle}\,ba \;,
\ee
where $\langle \;,\; \rangle$ denotes the usual scalar product, then the algebra $A$ is said to be \emph{$\Zn$-commutative}.

\begin{exas} Supercommutative algebras are the simplest examples $(\,$$n=1$$\,)$. As shown in \cite{MO1, MO2}, the quaternion algebra $\mathbb{H}$ $(\,$and more generally any Clifford algebra $\op{C}\!\ell_k$$\,)$ can be seen as a $\Zn$-commutative algebra for $n=3$ $(\,$respectively, $n=k+1$$\,)$. For this, it suffices to associate appropriate degrees to the generators, e.g.,
	$$ \deg(\mathsf{i})= (0,1,1)\;, \qquad \deg(\mathsf{j})= (1,0,1) \qquad (\mbox{ and thus } \deg(\mathsf{k})= (1,1,0)\;).$$
\end{exas}

A $\Zn$-commutative algebra $A$ has of course an \emph{underlying parity}, given by
$$ \Zn\ni \gamma=(\gamma_1, \ldots,\gamma_n)\; \mapsto\; \bar{\gamma}:=\sum_{k=1}^n \gamma_k \in\Zs \;.$$
In other words, degrees $\gamma$ in $\Zn$ are divided into \emph{even} ($\bar{\gamma}=\bar{0}$) and \emph{odd} ($\bar{\gamma}=\bar{1}$), which induces the analogous subdivision of the
homogeneous elements of $A$ (labeled in the same way as \emph{even} or \emph{odd}).\medskip 
%%%%%%%%%%%%%%%%%%%%
%%%%%%%%%%%%%%%%%%%%

Note that, following the generalised sign rule \eqref{signrule}, every odd-degree element of $A$ is nilpotent, as it is familiar in supergeometry. However, the higher $\Zn$-case ($n\ge 2$), is essentially different from the super case, as nonzero degree elements are not necessarily nilpotent -- more precisely, all even nonzero degree elements are not nilpotent.\medskip

Analogously to the definition of $\Zn$-commutative algebras, other notions of linear algebra are straightforwardly inferred.
In this way, $\Zn$-graded modules over a $\Zn$-commutative algebra $A$ and degree-preserving $A$-linear maps between them form an abelian category $\ModA$,
which naturally admits a \emph{symmetric monoidal structure}  $\ots_A$ (see, e.g., \cite[Section 2.1]{BPP:HAC}),
with braiding given by
$$
\begin{array}{rccc}
c_{VW}^{\op{gr}}: &V \otimes_A W &\to& W \otimes_A V\\
&v \otimes w &\mapsto& (-1)^{\langle \deg(v),\deg(w)\rangle} w \otimes v\;,
\end{array}
$$
for homogeneous elements $v$ and $w$.
This structure is also \emph{closed}, as for every $W\in\ModA$, the functor $-\ots_A W: \ModA \to \ModA $ has a right-adjoint $\uHom_A(W,-): \ModA \to \ModA$, i.e., for any graded $A$-modules $V,U$, there is a natural isomorphism
$$
\Hom_A(V\ots_A W, U)\simeq \Hom_A(V,\uHom_A(W,U))\;.
$$
As one can readily verify, the \emph{internal hom} $\uHom_A(V,W)$ is the graded $A$-module which consists of all $A$-linear maps
$  \ell: V \to W \;.$
These may shift the $\Zn$-degree of the elements by a fixed $\gamma\in \Zn$, i.e.,
$$
\ell(V^{\alpha})\subset W^{\alpha+\gamma}\;,
$$
for all $\alpha\in\Zn$. The latter constitute the $\gamma$-part $\uHom^{\gamma}_A(V,W)$ of $\uHom_A(V,W)$.
Hence, contrary to the case of modules over a classical commutative algebra, the internal hom $\uHom_A$ differs from the categorical hom $\Hom_A$, since this latter contains only $0$-degree $A$-linear maps. In other words, $\Hom_A(V,W)=\uHom^{0}_A(V,W)$.

\subsection{$\Zn$-supermanifolds and morphisms}

The basic objects of our study are smooth $\Zn$-supermanifolds. Aside from the usual commuting coordinates (denoted in the following with the letter $x$), they present also different ``types'' of ``formal'' coordinates $\xi$ (corresponding to the different nonzero degrees in $\Zn$) which may commute or anticommute, following the generalized sign rule \eqref{signrule}. It is important to note that, contrary to superspaces (case $n=1$), even coordinates may anticommute, odd coordinates may commute, and, as already mentioned, even nonzero degree coordinates are not nilpotent.\medskip

To keep track of these differences in a local coordinate system of a $\Zn$-supermanifold, we have to introduce some more notation.\medskip

Using the {underlying parity} we fix a \emph{standard order} of the elements of $\Zn$: first the even degrees ordered lexicographically, then the odd ones also ordered lexicographically. For example,
$$ \Zs^2=\{(0,0), (1,1), (0,1),(1,0)\}\;. $$
So, when we will refer to $\gamma_j\in\Zn\setminus\{0\}$, the \emph{$j$-th nonzero degree of $\Zn$}, we will always mean with respect to this standard order. We may thus write $\xi_{\gamma_j}$ to specify that the considered formal coordinates are of degree $\gamma_j\in\Zn\setminus\{0\}$.
Then, a tuple $\vect{q}=(q_1,\ldots,q_N)\in\R^N$ (where $N:=2^{n}-1$) provides all the information on the $\Zn$-graded variables $\xi$: there is a total of $|\vect{q}|:=\sum_{k=1}^N q_i$ graded variables $\xi^a$, among which $q_i$ of degree $\gamma_{i}\in\Zn\setminus\{0\}$, denoted by  $\xi_{\gamma_i}^{a_i}$ ($1\leq a_i\leq q_i$, $1\leq i\leq N$). For simplicity, the variables $\xi$ are also considered to be  ordered by degree. Hence, throughout this article, a system of coordinates of a $\Zn$-superspace will be denoted in different ways, depending on the level of distinction needed: either $u$ (no distinction between coordinates), $(x,\xi)$ (considering only the zero/nonzero degree subdivision) or $(x, \xi_{\gamma_j})$ (considering the full $\Zn$-degree subdivision).\medskip

We are now ready to recall the definition of a $\Zn$-supermanifold.
\begin{defn}
A \emph{locally $\Zn$-ringed space} is a pair $(M,\cO_M)$ of a topological space $M$ and a sheaf of $\Zn$-commutative $\R$-algebras over it, such that at every point $m\in M$ the stalk $\cO_{M,m}$ is a local graded ring.\smallskip

A (smooth) \emph{$\Zn$-supermanifold} of dimension $p|\vect{q}$ is a \emph{locally $\Zn$-ringed space} $(M,\cO_M)$ which is locally isomorphic to a $\Zn$-superdomain $(\R^p,\Ci_{\R^p}[[\xi]])$. Local sections of this latter are \emph{formal power series} in the $\Zn$-graded variables $\xi$ and smooth coefficients
$$
 \Ci(U)[[\xi]]:=\left\{ \sum_{\alpha\in\N^{N}}^{\infty} f_{\alpha}(x)\xi^{\alpha}\; |\; f_{\alpha}\in\Ci(U)\right\} \;.
$$

\emph{Morphisms} between $\Zn$-supermanifolds are simply morphisms of $\Zn$-ringed spaces, i.e., pairs
$(\phi,\phi^*):(M,\cO_M)\to (N,\cO_N)$ of a continuous map $\phi:M\to N$ and a sheaf morphism $\phi^*:\cO_N\to\phi_*\cO_M$, i.e., a family of algebra morphisms which are compatible with restrictions and are defined for any open $V\subset N$ by $$\phi^*_V:\cO_N(V)\to \cO_M(\phi^{-1}(V))\;.$$

We denote the category of $\Zn$-supermanifolds and morphisms between them by $\Zn{\tt Man}$.
\end{defn}

\begin{rema}
{\bf In the following, we will use $M$ both to denote the $\Zn$-supermanifold $(M,\cO_M)$ and its base space, preferring the classical notation $|M|$ for this latter when confusion can arise. And analogously we denote by $\phi$ the morphism $(\phi,\phi^*)$ and its base morphism, writing $|\phi|$ for the latter whenever necessary. Moreover, when considering sheaves, like, e.g., $\cO_M$, we omit the underlying topological space $M$, if this space is clear from the context.}\end{rema}

\begin{rema} Let us stress that the base $M$ corresponds to the degree zero coordinates (and not to the even degree coordinates), and let us mention that it can be proven that the topological base space $M$ carries a natural smooth manifold structure of dimension $p$, that the continuous base map $\phi:M\to N$ is in fact smooth, and that the morphisms $$\phi^*_m:\cO_{\phi(m)}\to\cO_{m},\;m\in M$$ between stalks induced by a morphism $\phi:M\to N$ of $\Zn$-supermanifolds respect the unique homogeneous maximal ideals of the local graded rings $\cO_{\phi(m)}$ and $\cO_{m}$.
\end{rema}

\subsection{$\cJ$-adic topology and Hausdorff completeness}
Let $I$ be a homogeneous ideal of a $\Zn$-graded (unital) ring $R$, $K$ an $R$-module. The collection of sets $\{x + I^kK \}_{k=0}^\infty$, where $x$ runs over all elements of $K$, is readily seen to be a basis for a topology on $K$, called the {\it $I$-adic topology}. It follows immediately from the definition that the $I$-adic topology is translation-invariant with respect to the additive group structure of $K$.\medskip

The following lemma is standard but we give its proof for completeness.

\begin{lem}\label{Iadiccont}
Let $I$ be a homogeneous ideal of $R$, and let $f: K \to L$ be an $R$-module morphism. Then $f$ is $I$-adically continuous.
\end{lem}

\begin{proof}
Let $y$ be any element of $L$. Since the sets $y + I^kL$ constitute a basis for the $I$-adic topology of $L$, it suffices to prove that $f^{-1}(y + I^kL)$ is $I$-adically open in $K$ for all $k$. As $f$ is an $R$-morphism, $f(I^kK) \subset I^kL$. Thus $f^{-1}(y + I^kL)$ is the union of the open sets $x + I^kK$ where $x$ runs over all elements of $f^{-1}(y + I^kL)$, hence it is open.
\end{proof}

The $I$-adic topology on $R$ makes $R$ into a {\it topological ring}, i.e., a ring such that the addition and multiplication maps $R \times R \to R$ are continuous, when $R \times R$ is endowed with the product topology. We will check the continuity of multiplication; continuity of addition is similar. Denote the multiplication map by $\mu$, and let $y$ be an element of $R$; then

\begin{align*}
\mu^{-1}(y + I^k) = \bigcup_{(a, b) \in \mu^{-1}(y + I^k)} \left(a + I^k \right) \times \left(b+ I^k \right).
\end{align*}
The sets $(a + I^k) \times (b + I^k)$ are open in the product topology on $R \times R$, hence $\mu^{-1}(y + I^k)$ is open, proving continuity of $\mu$.

\begin{defn}
Let $I$ be a homogeneous ideal of $R$. The ring $R$ is {\it Hausdorff complete} in the $I$-adic topology if the canonical ring morphism $p: R \;\to\; \varprojlim\nolimits_{k\in\N} R/I^k\;$ is an isomorphism.
\end{defn}

We next deal with convergence in the $I$-adic topology.

\begin{defn}
Let $I$ be a homogeneous ideal of $R$, and let $a_i$ be a sequence of elements of $R$. The sequence $a_i$ is an {\it $I$-adic Cauchy sequence} if for each non-negative integer $n$ there exists $l$ such that $a_j - a_k \in I^n$ for all $j, k \geq l$. The sequence $a_i$ {\it converges $I$-adically} if there exists $a \in R$ such that for each non-negative integer $n$, there exists $l$ such that $a_i - a \in I^n$ for all $i \geq l$.
\end{defn}

Evidently these definitions may be extended to the $I$-adic topology on an $R$-module $M$ in a natural way. The following proposition is standard.

\begin{prop}
Let $R$ be a $\Zn$-graded ring and $I$ a homogeneous ideal. Suppose $R$ is $I$-adically Hausdorff complete. Then a sequence $a_i$ is $I$-adically Cauchy if and only if it converges $I$-adically to a unique limit in $R$.
\end{prop}

\begin{proof} The fact that a convergent sequence is Cauchy is straightforward, so we prove the converse. If $a_i$ is a Cauchy sequence, it yields a well-defined element $[a]_k$ in $R/I^k$ for each $k$. Indeed, since $a_i \equiv a_j \text{ mod } I^k$ for $i, j$ sufficiently large, we define $[a]_k$ to be the equivalence class of $a_i$ for $i$ large enough. Noting that $$\varprojlim\nolimits_{k\in\N} R/I^k = \{(b_0, b_1, b_2, \dotsc ) \in \prod_{k \geq 0} R/I^k : b_k = f_{kl}(b_l), k \leq l\}\;, \text{where}\; f_{kl}: R/I^l\ni r+I^l \mapsto r+I^k\in R/I^k\;,$$ the $[a]_k$ define a unique element $a'$ of $\varprojlim\nolimits_{k\in\N} R/I^k$. Since $p$ is an isomorphism, this element can be identified with the corresponding element $a$ of $R$. It is then directly verified that $a$ is a limit of the sequence $a_i$. To prove uniqueness of the limit, suppose $a_i$ converges to $\widetilde{a}$. Then $a - \widetilde{a} = (a-a_i) + (a_i - \widetilde{a})$ lies in $I^k$ for all $k$, since $a_i$ converges to $a$ and to $\tilde a$. But since $p$ is an isomorphism, its kernel $\cap_{k =0}^\infty I^k$ vanishes, whence $a = \widetilde{a}$.
\end{proof}

A similar proposition is true for $I$-adically Hausdorff complete $R$-modules.\\

Canonically associated to any $\Zn$-graded algebra $R$ is the homogeneous ideal $J$ of $R$ generated by all homogeneous elements of $R$ having nonzero $\Zn$-degree. If $f: R \to S$ is a morphism of $\Zn$-graded algebras, then $f(J_R) \subset J_S$. The $J$-adic topology plays a fundamental role in $\Zn$-supergeometry.\medskip

Indeed, the preceding notions can be sheafified. For a $\Zn$-supermanifold $M$, we have an ideal sheaf $\cJ$, defined by $$\cJ(U) = \langle f \in \cO(U) \,:\, f \text{ is of nonzero $\Zn$-degree} \rangle\;.$$ The ideal sheaf $\cJ$ defines a $\cJ$-adic topology on the structure sheaf $\cO$ of $M$. If $\cF$ is a sheaf of $\cO$-modules, there is an analogous $\cJ$-adic topology on $\cF$. Throughout this paper, all statements about sheaves of topological $\cO$-modules (e.g., saying a sheaf is Hausdorff complete) will refer to this $\cJ$-adic topology.\medskip

As shown in \cite{Covolo:2016, Covolo:2016a}, many basic results valid for smooth $\Zs$-supermanifolds also hold for $\Zn$-supermanifolds. For instance, the underlying (base, or reduced) space $M$ of a $\Zn$-supermanifold admits a structure of smooth manifold $\Ci_M$, and there is a projection $\varepsilon:\cO_M\to\Ci_M$ of sheaves such that $\cJ =\ker\varepsilon$.\medskip

The obstacle, in the higher $\Zn$-case, represented by the loss of the nilpotency of $\cJ$ (a fundamental fact in supergeometry), is compensated by the Hausdorff completeness of the $\cJ$-adic topology on $\cO_M$:

\begin{prop}[Proposition 6.9 in \cite{Covolo:2016}]
Let $M$ be a $\mathbb{Z}^n_2$-supermanifold. Then $\cO_M$ is \emph{$\cJ$-adically Hausdorff complete} as a sheaf of $\mathbb{Z}^n_2$-commutative algebras, i.e., the canonical sheaf morphism
$$p: \cO_M \;\rightarrow\; \varprojlim\nolimits_{k\in\N} \cO_M/\cJ^k\;$$
is an isomorphism.
\end{prop}

\subsection{Functor of points}

Similar to what happens in classical $\Zs$-supergeometry, a $\Zn$-supermanifold $M$ is not fully described by its topological points in $|M|$. To remedy this defect, we broaden the notion of ``point", as was suggested by Grothendieck in the context of algebraic geometry.\medskip

More precisely, set $V=\{z\in\C^n:P(z)=0\}\in{\tt Aff}$, where $P$ denotes a polynomial in $n$ indeterminates with complex coefficients and $\tt Aff$ denotes the category of affine varieties. Grothendieck insisted on solving the equation $P(z)=0$ not only in $\C^n$, but in $A^n$, for any algebra $A$ in the category $\tt CA$ of commutative (associative unital) algebras (over $\C$). This leads to an arrow $$\op{Sol}_P:{\tt CA}\ni A\mapsto\op{Sol}_P(A)=\{a\in A^n:P(a)=0\}\in{\tt Set}\;,$$ which turns out to be a functor $$\op{Sol}_P\simeq \op{Hom}_{\tt CA}(\C[V],-)\in[{\tt CA},{\tt Set}]\;,$$ where $\C[V]$ is the algebra of polynomial functions of $V$. The dual of this functor, whose value $\op{Sol}_P(A)$ is the set of $A$-points of $V$, is the functor $$\Hom_{\tt Aff}(-,V)\in[{\tt Aff}^{\op{op}},{\tt Set}]\;,$$ whose value $\Hom_{\tt Aff}(W,V)$ is the set of $W$-points of $V$.\medskip

The latter functor can be considered not only in $\tt Aff$, but in any locally small category, in particular in $\Zn{\tt Man}$. We thus obtain a covariant functor (functor in $\bullet$) $$\underline{\bullet}=\Hom(-,\bullet): \Zn{\tt Man}\ni M\mapsto \underline{M}=\op{Hom}_{\Zn\tt Man}(-,M)\in [{\tt \Zn{\tt Man}^{\op{op}}, Set}]\;.$$ As suggested above, the contravariant functor $\Hom(-,M)$ (we omit the subscript $\Zn{\tt Man}$) (functor in $-$) is referred to as the {\it functor of points} of $M$. If $S\in\Zn{\tt Man}$, an {\it $S$-point} of $M$ is just a morphism $\pi_S\in\op{Hom}(S,M)$. One may regard an $S$-point of $M$ as a `family of points of $M$ parameterized by the points of $S$'. The functor $\underline{\bullet}$ is known as the {\it Yoneda embedding}. For any underlying locally small category $\tt C$ (here ${\tt C}=\Zn{\tt Man}$), the functor $\underline{\bullet}$ is fully faithful, what means that, for any $M,N\in\Zn{\tt Man}$, the map $$\underline{\bullet}_{M,N}:\Hom(M,N)\ni \phi\mapsto \Hom(-,\phi)\in\op{Nat}(\Hom(-,M),\Hom(-,N))$$ is bijective (here $\op{Nat}$ denotes the set of natural transformations). It can be checked that the correspondence $\underline{\bullet}_{M,N}$ is natural in $M$ and in $N$. Moreover, any fully faithful functor is automatically injective up to isomorphism on objects: $\underline{M}\simeq\underline{N}$ implies $M\simeq N$. Of course, the functor $\underline{\bullet}$ is not surjective up to isomorphism on objects, i.e., not every functor $X\in[\Zn{\tt Man}^{\op{op}},{\tt Set}]$ is isomorphic to a functor of the type $\underline{M}$. However, if such $M$ does exist, it is, due to the mentioned injectivity, unique up to isomorphism and it is called {\bf `the' representing $\Zn$-supermanifold} of $X$. Further, if $X,Y\in[\Zn{\tt Man}^{\op{op}},{\tt Set}]$ are two representable functors, represented by $M,N$ respectively, a morphism or natural transformation between them, provides, due to the mentioned bijectivity, a {\bf unique morphism between the representing $\Zn$-supermanifolds} $M$ and $N$. It follows that, instead of studying the category $\Zn{\tt Man}$, we can just as well focus on the functor category $[{\tt \Zn{\tt Man}^{\op{op}}, Set}]$ (which has better properties, in particular it has all limits and colimits).

%-%-%-%-%-%-%-%-%-%-%-%-%-%-%
\section{Tangent sheaf, tangent space, and tangent map}\label{sec:tg}
%-%-%-%-%-%-%-%-%-%-%-%-%-%-%
\subsection{Tangent sheaf}
The following definition is standard:

\begin{defn}
Let $f \in \cO(U)$. We denote by $W_f$ the open subset of $U$ made of all points $m \in U$ such that the restriction $f|_V = 0$ for some neighborhood $V$ of $m$. The {\it support} of $f$ is the closed subset of $U$ defined by $\op{supp}(f) := U \backslash W_f$.
\end{defn}

Let now $U$ be an open subset of a $\Zn$-supermanifold $M$. We consider the set $\op{Der}_\R(\cO(U))$ of $\mathbb{Z}^n_2$-graded $\R$-linear derivations of $\cO(U)$, i.e., of $\R$-linear maps $D: \cO(U) \to \cO(U)$, of all degrees $\deg(D)\in\Zn$, satisfying the graded Leibniz rule:
\[
D(ab) = Da \cdot b + (-1)^{\langle \deg(D), \deg(a) \rangle} a \cdot Db\;.
\]\smallskip
Then $\op{Der}_\R(\cO(U)$) is a $\Zn$-graded $\cO(U)$-module, as may readily be verified.\medskip

We now show that derivations can be localized.\medskip

The existence of partitions of unity on $\Zn$-supermanifolds \cite[Section 7.4]{Covolo:2016} implies, as usual, the existence of {\it bump functions} on $\Zn$-supermanifolds. This means that, for a $\Zn$-supermanifold $M$, a point $m\in M$, and an open neighborhood $V$ of $m$, there is a function $\varphi\in\cO^0(M)$ such that $\op{supp}(\varphi)\subset V$ and $\varphi=1$ on an open neighborhood of $m$. The proof of this statement is standard.

\begin{lem}
Let $M$ be a $\mathbb{Z}^n_2$-supermanifold, $D: \cO(M) \to \cO(M)$ a global derivation, and $U$ an open subset of $M$. Then there exists a unique derivation $D|_U: \cO(U) \to \cO(U)$ such that $D|_U(g|_U)=(Dg)|_U$ for any $g\in\cO(M)$. This derivation $D|_U$ has the same degree as $D$.
\end{lem}

\begin{proof}
Since we wish to localize $D$, we must first show that $D$ is a local operator, i.e., that if $h|_V = 0$, then $Dh|_V = 0$ for any open $V$. This result is of course a direct consequence of the existence of bump functions. Indeed, suppose $m \in V$ and choose a bump function $\varphi\in\cO^0(M)$ around $m$ whose support $\op{supp}(\varphi)$ is contained in $V$ and which is 1 in an open neighborhood $W$ of $m$. Then $\varphi h = 0$, whence $D(\varphi h) = D \varphi \cdot h + \varphi \cdot Dh = 0$. As $h|_W = 0$ and $\varphi|_W = 1$, we have that $Dh|_W = 0$. But this is true near any point $m \in V$, whence $Dh|_V = 0$.\medskip

As a derivation $D$ is thus a local operator, we can define its localization $D|_U$. Suppose $f \in \cO(U)$ and $m \in U$, and choose a bump function $\varphi$ around $m$. The function $h=\varphi f \in \cO(M)$ \cite{Leites:1980} agrees with $f$ in some open neighborhood $V$ of $m$ in $U$. By the above discussion, the function $(Dh)|_V$ is independent of the choice of $h$ that agrees with $f$ on $V$ and depends only on $f$. The same argument allows to see that the functions $(Dh)|_V$, which are defined on a cover of $U$ by open subsets $V$, piece together to define a unique function on $U$, which we denote by $D|_Uf\in\cO(U)$ and whose degree is $\deg(D)+\deg(f)$. This procedure defines an operator $D|_U:\cO(U)\to\cO(U)$, which is readily seen to be a graded derivation from the fact that $D$ is. The uniqueness of $D|_U$ is a straightforward consequence of our proof. Further, it follows from the definition of $D|_U$ that $D|_U(g|_U)=(Dg)|_U$.
\end{proof}

Hence, given $V \subset U$, we may define a restriction morphism $\rho_{UV}: \op{Der}_\R(\cO(U)) \to \op{Der}_\R(\cO(V))$ by assigning to any derivation $X$ on $U$ the unique derivation $X|_V$ on $V$ given by the preceding lemma. It is readily checked that the $\rho_{UV}$ so defined satisfy the axioms for the restriction morphisms of a sheaf of $\Zn$-graded $\cO$-modules. We denote this sheaf by $\op{Der}_\R\cO$. It is endowed with the $\cJ$-adic topology associated to the decreasing filtration $$\op{Der}_\R\cO\supset \cJ\cdot \op{Der}_\R\cO\supset \cJ^2\cdot \op{Der}_\R\cO\supset\ldots$$of $\Zn$-graded $\cO$-modules. This topology makes $\op{Der}_\R\cO$ a sheaf of $\Zn$-graded topological $\cO$-modules. Indeed, the just defined restriction maps are continuous. The usual argument goes through here: since $\rho_{UV}(\cJ^k(U)\cdot\op{Der}_\R(\cO(U)))\subset \cJ^k(V)\cdot\op{Der}_\R(\cO(V))$, the preimage $\rho_{UV}^{-1}(D_V+\cJ^k(V)\cdot\op{Der}_\R(\cO(V)))$, where $D_V\in\op{Der}_\R(\cO(V))$, is the union of the $\Delta_U+\cJ^k(U)\cdot\op{Der}_\R(\cO(U))$, where $\Delta_U\in\op{Der}_\R(\cO(U))$ runs through this preimage.
\begin{defn}
The {\it tangent sheaf} $\,\cT M$ of a $\mathbb{Z}^n_2$-supermanifold $M$ with structure sheaf $\cO$ is the sheaf of $\Zn$-graded topological $\cO$-modules
$$
\cT M(U) := \op{Der}_{\R}(\cO(U))\;
$$
with the restriction morphisms $\rho_{UV}$ defined above. The sections in $\cT M(U)$ are referred to as {\it vector fields} on $U$.
\end{defn}

The next property of derivations will be crucial in much of what follows:
\begin{prop}\label{dercont}
Any derivation in $\op{Der}_\R(\cO(U))$ over any open $U\subset M$ is $\cJ(U)$-adically continuous.
\end{prop}

\begin{rema}\label{ContRem} Just as the ideal sheaf $\cJ$ induces the $\cJ$-adic topology on the algebra sheaf $\cO$ $(\,$the ideal $\cJ(U)$ induces the $\cJ(U)$-adic topology on the algebra $\cO(U)$ -- for any open $U\subset M$$\,)$, the stalk $\cJ_m$ implements the $\cJ_m$-adic topology on the stalk $\cO_m$ -- for any $m\in M$. We can of course consider the derivations $\op{Der}_\R\cO_m$. As for $\op{Der}_\R(\cO(U))$, any derivation in $\op{Der}_\R\cO_m$ is $\cJ_m$-adically continuous. For any $X\in\op{Der}_\R(\cO(U))$ its continuity implies that if a sequence of sections $f_k\in\cO(U)$ tends $\cJ(U)$-adically to a section $f\in\cO(U)$, then $Xf_k$ tends $\cJ(U)$-adically to $Xf$. A similar statement holds at the level of stalks.\end{rema}

\begin{proof}
For $X\in\op{Der}_\R(\cO(U))$ and $k\in\N\setminus\{0\}$, we have $X (\cJ^k(U)) \subset \cJ^{k-1}(U)$. Indeed, the case $k = 1$ is vacuously true. Suppose $X(\cJ^k(U)) \subset \cJ^{k-1}(U)$ for some $k$. Any element of $\cJ^{k+1}(U)$ is a finite sum of elements of the form $ab$, with $a \in \cJ(U)$ and $b \in \cJ^k(U)$. Then $$X(ab) = Xa \cdot b + (-1)^{\langle \deg(X), \deg(a) \rangle} a \cdot Xb\in \cJ^k(U)\;.$$
Let now $g \in \cO(U)$. Again the fact that $X (\cJ^{k+1}(U)) \subset \cJ^k(U)$ implies that $X^{-1}(g + \cJ^k(U))$ is the union of the $\cJ(U)$-adically open sets $f + \cJ^{k+1}(U)$, where $f$ runs over all elements of $X^{-1}(g + \cJ^k(U))$, so that $X^{-1}(g + \cJ^k(U))$ is open.
\end{proof}

\begin{prop}The real $\Zn$-graded vector space $\cT M(U)$ of vector fields on $U$ carries a \emph{$\mathbb{Z}^n_2$-graded Lie algebra} structure, which is given by the $\Zn$-graded Lie bracket
$$
[X, Y](f) := X(Yf) - (-1)^{\langle \deg(X),\, \deg(Y) \rangle} Y(Xf),
$$
for homogeneous $X, Y \in \cT M(U)$ and any $f \in \cO(U)$.\end{prop}

Indeed, it is straightforwardly checked that the above defines an $\R$-bilinear degree respecting operation $[- \, , -]$, which is $\Zn$-graded antisymmetric and satisfies the $\Zn$-graded Jacobi identity.\medskip

\begin{prop}\label{tangentbasis}
Let $M$ be a $\,\mathbb{Z}^n_2$-supermanifold of dimension $p|{\bf q}\,$. Then $\cT M$ is a locally free rank $p|{\bf q}$ sheaf of $\Zn$-graded topological $\cO$-modules. More precisely, if $u=(u^i)$ is a coordinate system on an open set $U$, the partial derivatives $(\partial_{u^i})$ form an $\cO(U)$-basis of $\cT M(U)$.
\end{prop}

\begin{rema}
Consequently, the stalk $(\cT M)_m$ at any $m\in M$ is a free $\cO_m$-module of rank $p|{\bf q}$ with induced basis $\left( \left[\partial_{u^i} \right]_m \right) \,$.
\end{rema}

\begin{proof}
That the $\partial_{u^i}$ are $\cO(U)$-linearly independent is readily checked. To show that they span $\cT M(U)$, let $D$ be a derivation on $U$. Set $a^b:= Du^b$ and $D': = D - \sum_b a^b \partial_{u^b}$. Since $D'$ is a graded derivation that vanishes on the $u^b$, we get $D'P = 0$ for any polynomial $P$ in the $u^b$. By $\cJ(U)$-continuity, it follows that $D'\mathcal{P} = 0$ for any polynomial section in the sense of \cite{Covolo:2016}, i.e., for any section in $\cO(U)$ of the form $$\mathcal{P} = \sum_{|\mu| \geq 0} P_\mu(x) \xi^\mu\;,$$ where $\mu$ is a multi-index and $P_\mu(x)$ is a polynomial in the zero-degree coordinates $x^i$ (apply Remark \ref{ContRem} to the sequence $\mathcal{P}_k$ obtained by taking in the sum $\mathcal{P}$ only the terms $|\mu|\le k$). Let now $f \in \cO(U)$ and let $m$ be any point in $U$. By polynomial approximation \cite[Theorem 6.10]{Covolo:2016}, for any $k$ there exists a polynomial section $\mathcal{P}$ such that $[f]_m - [\mathcal{P}]_m \in \mathfrak{m}^k_m$, where $\mathfrak{m}_m$ denotes the unique homogeneous maximal ideal of $\cO_m$. Applying $D'$ to $f - \mathcal{P}$, we see that $[D'f]_m$ lies in $\mathfrak{m}^{k-1}_m$ for every $k$, hence $[D'f]_m = 0$. Since $m \in U$ was arbitrary, $D'f = 0$ for any function $f \in \cO(U)$, i.e., $D = \sum_b a^b \partial_{u^b}$. The point is here that the induced derivation $D'_m$ sends ${\frak m}_m^k$ to $\frak{m}_m^{k-1}$, what can be proven by induction on $k$ using the fact that $D'(fg)=(D'f)g\pm f D'g$. Indeed, to show that the induction starts for $k=2$, it suffices to choose $f,g\in\frak{m}_m$, and, to show that the statement holds for $k+1$ if it holds for $k\ge 2$, if suffices to choose $f\in\frak{m}_m,g\in\frak{m}^k_m$.
\end{proof}

\begin{cor}
The sheaf of modules $\cT M$ is Hausdorff complete. Likewise, at every point $m\in M$, the stalk $(\cT M)_m$ is Hausdorff complete.
\end{cor}

\begin{proof}
By Proposition \ref{tangentbasis}, for any coordinate set $U$, $\cT M(U)$ is a free $\cO(U)$-module of finite rank. Since the structure sheaf $\cO$ is $\cJ$-adically Hausdorff complete, we can conclude.
\end{proof}

%------------------------------------------------
\subsection{Tangent space and tangent map}
%------------------------------------------------

\begin{defn}
Let $M$ be a $\Zn$-supermanifold and let $m \in M$. The {\it tangent space} to $M$ at $m$, denoted $T_mM$, is the $\mathbb{Z}^n_2$-super $\R$-vector space $\op{Der}_{\R,m}\cO_m$ of graded $\R$-linear derivations $\cO_m \to \R$.
\end{defn}

\begin{prop}
Let $X$ be a vector field defined in a neighborhood of $m$. Then $X$ induces a tangent vector $X|_m$ to $M$ at $m$. If $X$ is homogeneous, the degree of $X|_m$ is the same as that of $X$.
\end{prop}

\begin{proof}
A vector field $X: \cO(U) \to \cO(U)$ defined in a neighborhood $U$ of $m$ induces a graded derivation $X_m: \cO_m \to \cO_m$ at the stalk level such that if $X$ is homogeneous, $X_m$ has the same degree as $X$. Let $\varepsilon_m: \cO_m \to \Ci_m$ be the algebra morphism (of degree $0$) induced by the algebra morphism $\varepsilon_U:\cO(U)\to\Ci(U)$ and let $\ev_m: \Ci _m \to \R$ be the evaluation algebra morphism (of degree $0$) at $m$. We set
\be\label{tgtvecfromfield}
(X|_m)[f]_m := (\ev_m \circ \varepsilon_m \circ X_m)[f]_m\;.
\ee
It is readily verified that $X|_m: \cO_m \to \R$ has all the announced properties.
\end{proof}

As with the tangent sheaf, if $\dim(M) = p|{\bf q}$ then $\dim(T_mM) = p|{\bf q}$. Indeed, given a coordinate system $(x^i, \xi^a)$ centered at $m$, the tangent vectors $\left(\partial_{x^i}|_m, \partial_{\xi^a}|_m\right)$ at $m$, induced by the coordinate vector fields following \eqref{tgtvecfromfield}, are a basis for $T_mM$. As in the case of the tangent sheaf, this is proven by polynomial approximation.

\begin{rema} It is worth remembering that $X|_m[f]_m=(\varepsilon\,Xf)(m)$. Since the evaluation at $m$ of $Xf\in\cO(U)$ is meaningless, we often just write $X|_m[f]_m=(Xf)(m)$ or even $X|_m[f]_m=Xf|_m$. In particular, if $u=(x,\xi)$, we have $$\partial_{u^b}|_m[f]_m=(\varepsilon\;\partial_{u^b}f)(m)=\partial_{u^b}f|_m\;.$$\end{rema}

We may compare the geometric fiber of the tangent sheaf at a point $m$ with the tangent space at $m$ defined above.
\begin{prop}
Let $m \in M$ be a point, $\mathfrak{m}_m$ the maximal ideal of $\cO_m$. Then, $$ T_m M \simeq (\cT M)_m / \left(\mathfrak{m}_m\cdot(\cT M)_m\right)$$
as $\mathbb{Z}^n_2$-super $\R$-vector space.
\end{prop}
\begin{proof}
This follows by unraveling the proof of the previous proposition.
\end{proof}

\begin{defn}
Let $\psi: M \to N$ be a morphism of $\Zn$-supermanifolds such that $\psi(m) = n$ for some point $m \in M$. The {\it tangent map} of $\psi$ at $m$ is the morphism of $\mathbb{Z}^n_2$-super vector spaces $d\psi_m: T_m M \to T_nN$ defined by
\[
d\psi_m(v)([f]_n) = v(\psi^*_m([f]_n))\;,
\]
for all $v\in T_mM$ and $[f]_n\in \cO_{N,n}$.
\end{defn}
It follows directly from this definition that
\begin{prop}\label{prop:tgtcomp}
Let $\psi: M\to N$ and $\phi:N\to S$ be two morphisms of $\Zn$-supermanifolds. Then, for any point $m\in |M|$,
\be\label{comptgtmap}
d(\phi\circ\psi)_m=d\phi_{|\psi|(m)} \circ d\psi_m \;.
\ee
\end{prop}

%------------------------------------------------
\subsection{Chain rule and modified Jacobian.}
%------------------------------------------------

We now establish the chain rule for $\Zn$-supermanifolds and use it to relate the tangent map to the $\Zn$-graded Jacobian matrix.

\begin{prop}\label{chainrule}
Let $U^{p|{\bf q}}, V^{r|{\bf s}}$ be $\mathbb{Z}^n_2$-superdomains, with coordinates
$u^a, v^b$ respectively. Let $\psi: U^{p|{\bf q}} \to V^{r|{\bf s}}$ be a morphism of $\Zn$-supermanifolds. Then,
\[
\frac {\partial \psi^*(f)} {\partial u^a} = \sum_b \frac {\partial \psi^*(v^b)} {\partial u^a} \, \psi^* \bigg( \frac {\partial f} {\partial v^b} \bigg)
\]
for any $f \in \cO(V^{r|{\bf s}})$.
\end{prop}

\begin{proof}
Let $f \in \cO(V^{r|{\bf s}})$,
\beas
D(f) := \frac {\partial \psi^*(f)} {\partial u^a}\;,
&& D'(f) := \sum_b \frac {\partial \psi^*(v^b)} {\partial u^a} \, \psi^* \bigg( \frac {\partial f} {\partial v^b} \bigg)\;,
\eeas
and consider the graded derivation $D - D'$. Clearly, $D - D'$ annihilates all polynomials in the coordinate functions $v^b$, as it is a graded derivation. By Proposition \ref{dercont}, $D - D'$ annihilates all polynomial sections on $V^{r|{\bf s}}$. Let $f \in \cO(V^{r|{\bf s}})$ and $m$ be any point in $V^{r|{\bf s}}$. By polynomial approximation \cite[Thm. 6.10]{Covolo:2016}, for any $k$ there exists a polynomial section $\mathcal{P}$ such that $[f]_m - [\mathcal{P}]_m \in \mathfrak{m}^k_m$. Since this is true for any $k$, we see that $[(D - D')f]_m = 0$, and since $m \in V^{r|{\bf s}}$ was arbitrary, we conclude that $(D - D')f = 0$.
\end{proof}

In particular, this implies that
\beas
d\psi_m\;(\partial_{u^a}\vert_m)
&=&
\sum_{b} \left(\partial_{u^a}\psi^*(v^b)\vert_m\right) \partial_{v^b}\vert_{|\psi|(m)}\;,
\eeas
so that the tangent map of $\psi$ at $m$ is the linear map characterized in the bases of its source and target tangent spaces induced by the coordinates $u$ and $v$ respectively, by the \emph{block-diagonal} $\Zn$-graded matrix
\be\label{MatTang}
 J_{\psi,m} := \left( \partial_{u^a}v^b|_m \right)_{ba}\;.
\ee
Note that, as in supergeometry, we simply wrote $v=v(u)$ instead of $\psi^*(v)$. The preceding result is a consequence of the above definitions and conventions, provided one remembers that $\varepsilon_m$ and $\psi^*_m$ commute \cite{Covolo:2016} and that $\op{ev_m}\psi^*_m=\op{ev}_{|\psi|(m)}$. As for the statement that $J_{\psi,m}$ is block-diagonal, note that if a non-diagonal block is non-zero, there exists a derivative $\p_{u^a}v^b$, with $\deg(v^b)\neq\deg(u^a)$, which contains a term without formal parameter. Hence, the corresponding term in $v^b$ contains the unique parameter $u^a$ and its degree is thus $\deg(v^b)=\deg(u^a)$, what is a contradiction.

\medskip
When considering the counterpart of \eqref{comptgtmap} in matrix form, we remark that an additional sign appears. Indeed, for $\Zn$-morphisms $\psi:M\to N$ and $\phi:N\to S$, which read locally respectively as $v=v(u)$ around $m$ and $w=w(v)$ around $|\psi|(m)$, one has
$$
\partial_{u^b}w^a= \sum_c \partial_{u^b}v^c \, \partial_{v^c}w^a
= \sum_c  (-1)^{\langle \deg(u^b)+\deg(v^c),\,\deg(v^c)+\deg(w^a)\rangle} \partial_{v^c}w^a \, \partial_{u^b}v^c \;.
$$

To absorb the redundant sign, we must consider the following
\begin{defn}
The \emph{$\Zn$-graded Jacobian matrix} of a local $\Zn$-morphism $\psi$ between $\Zn$-domains $U$ and $V$, given by $v=v(u)$, is the degree $0$ graded matrix
\be\label{GradModJac}
\op{Jac}_\psi:= \left( (-1)^{\langle \deg(v^b)+\deg(u^a),\,\deg(u^a) \rangle} \partial_{u^a}v^b\right)_{ba}\;.
\ee
\end{defn}

We then have the following familiar result.
\begin{prop}\label{JacobianComposition}
The graded Jacobian matrix of a composition of morphisms is the product of the graded Jacobians of the individual morphisms, i.e.,
\[
\op{Jac}_{\psi \circ \phi} = \op{Jac}_\psi \cdot \op{Jac}_\phi\;,
\]
or, more precisely,
\[
\op{Jac}_{\psi \circ \phi} = \phi^*(\op{Jac}_\psi) \cdot \op{Jac}_\phi\;
\]
\end{prop}

\begin{rema}\label{rem:tgmap-Jac}
Notice again that the Jacobian matrix $\op{Jac}_\psi$ of a $\Zn$-morphism $\psi$ is the usual Jacobian matrix with an \emph{extra sign} that depends on the considered entry (see Equation \ref{GradModJac}). On the other hand, the matrix $J_{\psi,m}$ of the tangent map $d\psi_m$ of a $\Zn$-morphism $\psi$ is the usual Jacobian matrix with derivatives that are (projected onto the base by $\varepsilon$ and then) evaluated at $m$ (see Equation \ref{MatTang}). As mentioned already above, the projection transforms the Jacobian matrix into a block-diagonal one. Observe that the difference in sign between $\op{Jac}_\psi$ and $J_{\psi,m}$ disappears for the diagonal blocks (see Equation \ref{GradModJac}). Therefore, the matrix $J_{\psi,m}$ can be viewed as the (projection and) evaluation at $m$ of the matrix $\op{Jac}_\psi$:
\be\label{tgmapmtrx}
J_{\psi,m}= \op{Jac}_{\psi}\big|_m =
		\left(
			\begin{array}{c|c|c|c}
					B_{\gamma_0} & & & \\
					\hline
					& B_{\gamma_1} & & \\
					\hline
					& & \ddots & \\
					\hline
					& & & B_{\gamma_{2^n-1}}
			\end{array}					
		\right),\quad B_{\gamma_i}=\left(\p_{u^a}v^b|_m\right)_{ba},\;\text{with}\; \deg(u^a)=\deg(v^b)=\gamma_i\in\Zn\;.
\ee
Moreover, the coordinates $u$ and $v$ are ordered according to the standard ordering, so that the $\gamma_i\in\Zn$ are ordered similarly, i.e., they run first through the even degrees ordered lexicographically, then through the odd ones ordered also lexicographically.
\end{rema}

%-%-%-%-%-%-%-%-%-%-%-%-%-%-%-%-%-%-%-%
\section{Products of $\Zn$-supermanifolds}\label{sec:product}
%-%-%-%-%-%-%-%-%-%-%-%-%-%-%-%-%-%-%-%

The category of $\Zn$-supermanifolds admits finite products:

\begin{prop}
Let $M_i$, $i\in\{1,2\}$, be $\Zn$-supermanifolds. Then there exists a $\Zn$-supermanifold $M_1 \times M_2$ and $\Zn$-morphisms $\pi_i: M_1 \times M_2 \to M_i$ (with underlying smooth manifold $|M_1\times M_2|=|M_1|\times |M_2|$ and with underlying smooth morphisms $|\pi_i|:|M_1|\times|M_2|\to|M_i|$ given by the canonical projections), such that for any $\Zn$-supermanifold $N$ and $\Zn$-morphisms $f_i: N \to M_i$, there exists a unique morphism $h:=\langle f_1,f_2\rangle: N \to M_1 \times M_2$ making the diagram

\begin{equation*}
\begin{tikzcd}
N \arrow[bend left]{drr}{f_1}
\arrow[bend right]{ddr}[swap]{f_2}
\arrow[dotted]{dr}{h=\langle f_1,f_2\rangle} & &\\
& M_1 \times M_2 \arrow{r}{\pi_1} \arrow{d}{\pi_2} & M_1\\
&M_2
\end{tikzcd}
\end{equation*}

\noindent commute.
\end{prop}

Of course, such a product is unique up to unique isomorphism.

\begin{proof} See \cite{Bruce:2019}.\end{proof}

\begin{cor} The category $\Zn\text{-}\mathtt{Man}$ admits finite products.\end{cor}

%-%-%-%-%-%-%-%-%-%-%-%-%-%-%
\section{Local forms of morphisms}
%-%-%-%-%-%-%-%-%-%-%-%-%-%-%

%------------------------------------------------
\subsection{Inverse function theorem} \label{ssec:IFT}
%------------------------------------------------

We begin with the $\Zn$-analogue of the inverse function theorem, which plays as fundamental a role in $\Zn$-supergeometry as the classical version does in ungraded geometry.

\begin{thm}\label{InvFunThm}
Let $\phi: M \to N$ be a morphism of $\Zn$-supermanifolds, and $m \in |M|$ be a point. Then the following are equivalent:
\begin{enumerate}
\item $d\phi_m$ is invertible,
\item there exist coordinate charts $U$ about $m$ and $V$ about $|\phi|(m)$ such that $\phi|_U: U \to V$ is a $\Zn$-diffeomorphism.
\end{enumerate}
\end{thm}

\begin{proof}
That $\phi$ being a $\Zn$-diffeomorphism around $m$ implies invertibility of $d\phi_m$ is an easy consequence of the chain rule.

Now let us suppose that $d\phi_m$ is invertible. In particular, this implies $\dim(M) = \dim(N)$. As the statement is local, we may assume from the beginning that $M$ and $N$ are $\Zn$-superdomains $U^{p|{\bf q}}$ and $V^{p|{\bf q}}$ respectively, and that $m = 0$. Let $\mu \in \Zn \backslash \{0\}$ and let $(x, \xi)$ and $(y, \eta)$ be coordinates on $U$ and $V$. In order to keep track of the $\Zn$-degrees of the coordinates, we have introduced in Section \ref{sec:prelim} the notation $\xi^k_\mu$ to indicate a coordinate of $\Zn$-degree $\mu$.

By the Chart Theorem for $\Zn$-supermanifolds \cite[Thm. 6.8]{Covolo:2016}, we have:
\begin{align}\label{phimorph}
&\phi^*(y^r) = f_0^r(x)+\sum_{\substack{|P| \ge 2\\ \deg(\xi^P)=0}} f_P^r(x)\, \xi^P\; ,\\
&\phi^*(\eta^s_\mu) = \sum_{\substack{|Q| \geq 1 \\ \deg(\xi^Q)=\mu}} g_{Q}^s(x)\, \xi^Q
										= \sum_{k=1}^{q_{\mu}} g_{\mathcs{E}(\mu,k)}^s(x) \,\xi^k_{\mu} + \ldots\label{phimorph2}
\end{align}
where, for $\mu$ the $i$-th nonzero degree in $\Zn$ following the standard order, $\mathsc{E}(\mu,k)=(0,\ldots, 1, \ldots, 0)$ with $1$ at the entry $q_1+ \ldots + q_{i-1}+k$.
Then, recalling the matrix form of the differential of $\phi$ at $m$ \eqref{tgmapmtrx}, the hypothesis that the differential be invertible is equivalent to the assumption that the block matrices

\begin{align*}
B_0
= \begin{pmatrix}
\frac {\partial f^1_0} {\partial x^1} & \dotsc & \frac {\partial f^1_0} {\partial x^p} \\
\vdots & & \vdots\\
\frac {\partial f^p_0} {\partial x^1} & \dotsc & \frac {\partial f^p_0} {\partial x^p}
\end{pmatrix}
\hspace{4mm}\text{and}\hspace{4mm}
B_\mu = \begin{pmatrix}
g^1_{\mathsc{E}(\mu,1)} & \dotsc & g^1_{\mathsc{E}(\mu, q_\mu)}  \\
\vdots & & \vdots\\
g^{q_\mu}_{\mathsc{E}(\mu,1)} & \dotsc & g^{q_\mu}_{\mathsc{E}(\mu, q_\mu)}
\end{pmatrix}
\end{align*}\\

\noindent are invertible for all $\mu \in \Zn \backslash \{0\}$. By the classical inverse function theorem, there exists an open neighborhood of $0$ ($m=0$) in $U$, such that the $\{ f^i_0 \}$ define new coordinates of degree $0$, denoted by $\widetilde{x}$ (we will denote this neighborhood by $U$, thus shrinking the original $U$ (and $V$) if necessary). Furthermore, using the invertible matrices $B_{\mu}$ we can also change the coordinates of each nonzero degree $\mu$, by
$$
\widetilde{\xi}^a_{\mu}:= \sum_{b} g^a_{\mathsc{E}(\mu,b)}(x) \,\xi^b_{\mu}\;,
$$
finally obtaining a new system of coordinates $(\widetilde{x}, \widetilde{\xi}\,)$ near the considered point.

Reading the morphism $\phi$, see \eqref{phimorph} and \eqref{phimorph2}, in the new coordinates $(\widetilde{x}, \widetilde{\xi}\,)$ and identifying these with the original coordinates $(x,\xi)$ to simplify notation, we get
\begin{align*}
&\phi^*(y^r) = x^r + \sum_{|P| \geq 2} \widetilde{f}^r_P(x)\,\xi^P=x^r + S^r(x,\xi)\; ,\\
&\phi^*(\eta^s_\mu) = \xi^s_\mu + \sum_{|Q| \geq 2} \widetilde{g}^s_{Q} (x)\,\xi^Q = \xi^s_\mu + \Sigma_\mu^s(x,\xi)\; ,
\end{align*}
where $S^r(x,\xi),\Sigma_\mu^s(x,\xi)\in\mathcal{J}^2(U)$ (in view of the identification, we have $V=U$).

Define now a $\Zn$-morphism $\psi:V^{p|\mathbf{q}}\to U^{p|\mathbf{q}}$ by setting
$$
\psi^*(x^i)=y^i\in\cO^0(U)\quad\text{and}\quad\psi^*(\xi^j_\mu)=\eta^j_\mu\in\cO^{\mu}(U)\;.
$$
The $\Zn$-morphism $\psi\circ\phi:U^{p|\mathbf{q}}\to U^{p|\mathbf{q}}$ is characterized by the pullbacks
$$
\phi^*(\psi^*(x^i))=x^i+S^i(x,\xi)\quad\text{and}\quad \phi^*(\psi^*(\xi^j_\mu))=\xi^j_\mu+\Sigma_\mu^j(x,\xi)\;.
$$
For any $\Zn$-function $h(x,\xi)=\sum_\alpha h_\alpha(x)\xi^\alpha\in\cO(U)$, the pullback by $\psi\circ\phi$ is
$$
(\phi^*\circ\psi^*)(h(x,\xi))= \sum_\alpha h_\alpha(x+S)(\xi+\Sigma)^\alpha = \sum_\alpha\;\sum_\beta\frac{1}{\beta!}\,(\p_x^\beta h_\alpha)(x)\,S^\beta(x,\xi)\;(\xi+\Sigma)^\alpha\;,
$$
where $S^\beta(x,\xi)\in\mathcal{J}^{2|\beta|}(U)$. Notice that
$$
(\xi+\Sigma)^\alpha= \xi^\alpha + {\tt S}_\alpha(x,\xi)\quad\text{and}\quad \sum_\beta\ldots = h_\alpha(x)+\sum_{|\beta|\ge 1}{\sf S}_{\alpha\beta}(x,\xi)\;,
$$
where ${\tt S}_\alpha(x,\xi)\in\mathcal{J}^{|\alpha|+1}(U)$ and ${\sf S}_{\alpha\beta}(x,\xi)\in\mathcal{J}^{2|\beta|}(U)$. Hence,
$$
(\phi^*\circ\psi^*)(h(x,\xi)) = h(x,\xi) + \sum_\alpha h_\alpha(x)\,{\tt S}_\alpha(x,\xi)+\sum_\alpha\sum_{|\beta|\ge 1}{\sf S}_{\alpha\beta}(x,\xi)\,\xi^\alpha + \sum_\alpha \sum_{|\beta|\ge 1}{\sf S}_{\alpha\beta}(x,\xi)\,{\tt S}_\alpha(x,\xi)\;.
$$
If now $h(x,\xi)\in\mathcal{J}^k(U)$ ($k\ge 0$), i.e., if $|\alpha|\ge k$, then the four terms of the right hand side of the preceding equation belong to $\mathcal{J}^k(U)$, $\mathcal{J}^{k+1}(U)$, $\mathcal{J}^{k+2}(U)$, and $\mathcal{J}^{k+3}(U)$, respectively. This implies that
$$
(\phi^*\circ\psi^*)(h(x,\xi)) = h(x,\xi) + \mathfrak{S}(h(x,\xi))\;,
$$
with $\mathfrak{S}(h(x,\xi))\in\mathcal{J}^{k+1}(U)$. Hence, we have a degree 0 linear map (not a $\Z_2^n$-morphism)
$$
\mathfrak{S}:\mathcal{O}(U)\to\mathcal{J}(U)\subset\cO(U)\;,
$$
such that $\mathfrak{S}:\mathcal{J}^k(U)\to\mathcal{J}^{k+1}(U)$, for all $k\ge 1$. Define now
$$
\iota^*=\sum_{k=0}^\infty(-1)^k\mathfrak{S}^k=\id - \mathfrak{S}+\mathfrak{S}^2-\ldots
$$
For any $F\in\mathcal{O}(U)$, the successive terms $F$, $\mathfrak{S}F$, $\mathfrak{S}^2F$, ... of  $\iota^*(F)$ are elements of $\cO(U)$, $\mathcal{J}(U)$, $\mathcal{J}^2(U)$, ... This means that they are formal power series with at least 0, 1, 2, ... parameters. It follows that the series of series $\iota^*(F)$ is itself a formal power series in the $\xi$-s with coefficients in the smooth functions of the base $U$. Indeed, whatever the monomial $\xi^\alpha$, the coefficients of the terms of $\mathfrak{S}^\ell F$ with $\ell\ge |\alpha|+1$ cannot contribute to the coefficient of $\xi^\alpha$, which is thus a {\it finite} sum of base functions. Therefore $\iota^*$ is a map
$$
\iota^*:\cO(U)\to\cO(U)\;.
$$
Further, it is clear that $(\phi^*\circ\psi^*)\circ\iota^*=\iota^*\circ(\phi^*\circ\psi^*)=\id$, so that $\iota^*$ is a $\Zn$-graded unital algebra endomorphism of $\cO(U)$. Such an algebra morphism defines a unique $\Zn$-morphism $\iota:U^{p|\mathbf{q}}\to U^{p|\mathbf{q}}$ \cite{Bruce:2019}. Since $\phi^*\circ(\psi^*\circ\iota^*)=\id$ the $\Zn$-morphism $\lambda:=\iota\circ\psi$ is a left inverse of the $\Zn$-morphism $\phi$: $\lambda\circ\phi=\id$. In view of Proposition \ref{prop:tgtcomp}, we find that $d\lambda_{|\phi|(m)}\circ d\phi_m=\id$, so that $d\lambda_{|\phi|(m)}$ is invertible. Applying the result we just proved to $\lambda$, we find a local left inverse $\varphi$ of $\lambda$: $\varphi\circ\lambda=\id$, hence, $\varphi=\phi$ and $\phi$ is a local $\Zn$-diffeomorphism.
\end{proof}

%------------------------------------------------
\subsection{Implicit function theorem} \label{ssec:ImFcTh}
%-

As a corollary of the inverse function theorem on $\Zn$-manifolds, one can prove the \emph{implicit function theorem}.

\begin{cor}[Implicit Function Theorem]
Let $U:=U^{p|\mathbf{q}},V:=V^{r|\mathbf{s}}$ and $W:=W^{r|\mathbf{s}}$ be $\Zn$-domains with coordinates $$(U,u=(x,\xi_{\mu})),\, (V,v=(y,\eta_{\mu}))\;\text{and}\; (W,w=(z,\theta_{\mu}))\;.$$ Let $$\phi: U\times V \to W$$ be a $\Zn$-morphism and let $u_0\in |U|, v_0\in |V|$ and $w_0:=|\phi|(u_0,v_0)\in|W|$. Assume the graded submatrix $$J:=\partial_v \phi^*(w)|_{(u_0,v_0)}$$ of $\op{Jac}_{\phi}|_{(u_0,v_0)}$ is invertible. Then, it exists a unique $\Zn$-morphism $$\psi: U\times W \to V $$ defined in a neighborhood of $(u_0,w_0)$ (we shrink the original domains if needed), such that \be\label{ImplFuncBasis}|\psi|(u_0,w_0)=v_0\ee and
\be\label{ImplFuncEquation} \phi\circ \langle \pi_1^{U\times W},\psi\rangle = \pi_2^{U\times W}\;.\ee
\end{cor}

Obviously $\pi_1^{U\times W}$ and $\pi_2^{U\times W}$ refer to the first and second projections of the product $U\times W$. The morphism $\langle\pi_1^{U\times W},\psi\rangle$ is the unique $\Zn$-morphism from $U\times W$ to $U\times V$ that is defined by $\pi_1^{U\times W}$ and by $\psi.$ In the ungraded context, Equation \eqref{ImplFuncEquation} reads $\phi(u,\psi(u,w))=w$. If we fix $w$ setting $w=\phi(u_0,v_0)$, we get $$\phi(u,v)=\phi(u_0,v_0)\Leftrightarrow v=\psi(u)\;.$$

\begin{proof}
	Set $\chi := \langle\pi_1^{U\times V},\phi\rangle$. This $\Zn$-morphism $\chi: U\times V \to U\times W$ is defined by
	\beas
	\chi^*(u)=u & \mbox{ and }& \chi^*(w)=\phi^*(w) \;.
	\eeas
By direct computation, we have that $\op{Jac}_{\chi}|_{(u_0,v_0)}$ is a block-diagonal matrix (see \eqref{tgmapmtrx}) with blocks
\beas
B_0= \left(\begin{array}{cc}
					\I & 0 \\
					\partial_x\phi^*(z)|_{(u_0,v_0)} & \partial_y\phi^*(z) |_{(u_0,v_0)}
			\end{array}\right)
& \mbox{ and }&
B_{\mu}= \left(\begin{array}{cc}
					\I & 0 \\
					 \partial_{\xi_{\mu}}\phi^*(\theta_{\mu})|_{(u_0,v_0)} & \partial_{\eta_{\mu}}\phi^*(\theta_{\mu}) |_{(u_0,v_0)}
			\end{array}\right)\;.
\eeas
Note that the block-diagonal matrix formed by the blocks
$\partial_y\phi^*(z) |_{(u_0,v_0)}$ and $\partial_{\eta_{\mu}}\phi^*(\theta_{\mu})|_{(u_0,v_0)}$ ($\mu\in \Zn\setminus\{0\}$) is nothing but the graded matrix $J$, which is by assumption invertible.
Thus, $\op{Jac}_{\chi}$ is invertible at $(u_0,v_0)$ and by Theorem \ref{InvFunThm} this implies that $\chi$ admits an inverse defined on a neighbourhood of $(u_0,w_0)$,
$$ \chi^{-1}: U\times W \to U\times V \;.$$

The $\Zn$-morphism $\psi:U\times W \to V$ defined by $\psi:= \pi_2^{U\times V}\circ\chi^{-1}$, has the properties described in the claim. Indeed, on the one hand, we have $|\psi|(u_0,w_0)=v_0$. On the other hand, the morphism $\langle\pi_1^{U\times W},\psi\rangle$ is characterized by $$\pi_1^{U\times V}\circ\langle\pi_1^{U\times W},\psi\rangle=\pi_1^{U\times W}\quad\text{and}\quad \pi_2^{U\times V}\circ\langle\pi_1^{U\times W},\psi\rangle=\psi\;.$$ Since, by definition of $\chi$, $$\pi_1^{U\times W}=\pi_1^{U\times V}\circ\chi^{-1}\quad\text{and}\quad\pi_2^{U\times W}=\phi\circ\chi^{-1}\;,$$ and, by definition of $\psi$, $$\psi= \pi_2^{U\times V}\circ\chi^{-1}\;,$$ we get
$$ \langle\pi_1^{U\times W},\psi\rangle = \chi^{-1}\;,$$ and $$\phi\circ\langle\pi_1^{U\times W},\psi\rangle=\pi_2^{U\times W}\;.$$

It remains to prove uniqueness of $\psi$. Assume there is a morphism $\omega:U\times W\to V$ that has the properties \eqref{ImplFuncBasis} and \eqref{ImplFuncEquation}. If we prove that $\p_w\omega^*(v)|_{(u_0,w_0)}$ is invertible, then the argument used above for $\phi$ implies that $\langle \pi_1^{U\times W},\omega\rangle$ is invertible, so that its inverse satisfies the characterizing properties of $\chi$: $$\langle \pi_1^{U\times W},\omega\rangle = \langle \pi_1^{U\times V},\phi\rangle^{-1}\;.$$ Further, by definition of $\langle \pi_1^{U\times W},\omega\rangle$, we have necessarily: $$\omega=\pi_2^{U\times V}\circ\langle \pi_1^{U\times V},\phi\rangle^{-1}\;.$$

Invertibility of $\p_w\omega^*(v)|_{(u_0,w_0)}$ is a consequence of \eqref{ImplFuncEquation} and the facts $$\left(\pi_2^{U\times W}\right)^*\!(w)=w,\; \langle \pi_1^{U\times W},\omega\rangle^*(u)=u\;\,\text{and}\;\,\langle \pi_1^{U\times W},\omega\rangle^*(v)=\omega^*(v)\;.$$ Indeed, using Equation \eqref{ImplFuncEquation} and Proposition \ref{JacobianComposition}, we get $$\left(0\;\,\I\right)=\left(\p_u\phi^*(w)\;\,\p_v\phi^*(w)\right)|_{(u_0,v_0)}\cdot\left(
                                                                                                                       \begin{array}{cc}
                                                                                                                         \I & 0 \\
                                                                                                                         \p_u\omega^*(v) & \p_w\omega^*(v) \\
                                                                                                                       \end{array}                                                                                                              \right)|_{_{(u_0,w_0)}}\;.$$
It follows that $$\p_w\omega^*(v)|_{(u_0,w_0)}=\p_v\phi^*(w)|^{-1}_{(u_0,v_0)}\;.$$
\end{proof}

%%%%%%%%%%%%%%%%%%%%%

%-----------------------------------------------
\subsection{Immersions and submersions}\label{ImmSub}
%------------------------------------------------
\begin{defn}
A $\Zn$-morphism $\phi:M\to N$ is an \emph{immersion} (resp., a \emph{submersion}) at a point $m\in |M|$, if its tangent map at this point $d\phi_m$ is injective (resp., surjective).
\end{defn}
Since this condition is a local property, we can replace $\phi$ in the tangent map by its restriction $\phi: U \to V$, where $U=U^{p|\mathbf{q}}$ and $V=V^{r|\mathbf{s}}$ are $\Zn$-domains with coordinate systems $u=(x,\xi_{\mu})$ and $v=(y,\theta_{\mu})$, respectively, and such that $m\in |U|$ and $|\phi|(m)\in |V|$.

The differential $d\phi_m$ is a linear map between $\Zn$-graded vector spaces, which is represented by a block-diagonal matrix $(B_{\gamma})_{\gamma}$ (see \eqref{tgmapmtrx}) in $\gl^0(r|\mathbf{s}\times p|\mathbf{q},\R)$. Hence, it can be seen as a collection of $2^n$ classical linear maps between real vector spaces represented by the classical real matrices $B_{\gamma}$, $\gamma \in \Zn$. Thus, $d\phi_m$ is injective (resp., surjective) if and only if each of these classical linear maps is injective (resp., surjective), i.e., the rank of each $B_{\gamma}$ is equal to the number of its columns (resp., rows).
It is then legitimate to define the {\it rank} of $d\phi_m$ as the collection of the ranks of its diagonal blocks, so that we can rephrase the above as:
\begin{itemize}
\item $\phi$ is an immersion at $m$ if and only if $\op{rank}(d\phi_m)=p|\mathbf{q}$\;,
\item $\phi$ is a submersion at $m$ if and only if $\op{rank}(d\phi_m)=r|\mathbf{s}$\;.
\end{itemize}

We have moreover the following results:
\begin{prop}\label{prop:sub-im}
Let $\phi:M\to N$ be a $\Zn$-morphism, $M$ of dimension $p|\mathbf{q}$ and $N$ of dimension $r|\mathbf{s}$, $m$ a point of $M$.
	\begin{enumerate}
			\item\label{immersion} Let $p\leq r$ and $q_j\leq s_j$, for all  $j$. The map $\phi$ is an immersion at $m$ if and only if there exists $\Zn$-charts $(U,(x,\xi_{\mu}))$ around $m$ and $(V,(y,\theta_{\mu}))$ around $|\phi|(m)\in |N|$, such that $\phi|_U:U\to V$ has the form
			\beas
			\phi^*_V(y^i)=\begin{cases}
														x^i &\mbox{ if } 1\leq i\leq p \\
														0 &\mbox{ if } p+1\leq i\leq r
										\end{cases}
			& \mbox{ and } &
			\phi^*_V(\theta_{\mu}^{a})=\begin{cases}
																					\xi_{\mu}^{a} &\mbox{ if } 1\leq a\leq q_{\mu} \\
																					0 &\mbox{ if } q_{\mu}+1\leq a\leq s_{\mu}
																				\end{cases}\;,
			\eeas
			for all $\mu\in\Zn\setminus\{0\}$. Here we denote by $q_{\mu}$ the entry of $\mathbf{q}$ corresponding to the degree $\mu$, and analogously for $s_{\mu}$.
			
			\item\label{submersion} Let $p\geq r$ and $q_j\geq s_j$, for all $j$. The map $\phi$ is a submersion at $m$ if and only if it exists $\Zn$-charts $(U,(x,\xi_{\mu}))$ around $m$ and $(V,(y,\theta_{\mu}))$ around $|\phi|(m)\in |N|$, such that $\phi|_U:U\to V$ has the form
			\bea\label{submdef}
			\phi^*_V(y^i)= x^i \,,\quad 1\leq i\leq r
			& \mbox{ and } &
			\phi^*_V(\theta_{\mu}^{a})= \xi^{a}_{\mu}\;, \quad 1\leq a\leq s_{\mu}\;,
			\eea
			for all $\mu\in\Zn\setminus\{0\}$.
	\end{enumerate}
\end{prop}
In other words, $\phi$ is an immersion (resp., a submersion) at $m$ if and only if in a neighborhood of $m$ and $|\phi|(m)$ there exist coordinates in which $\phi$ is the canonical linear injection (resp., linear projection).

\begin{proof}
The proofs of the two claims \eqref{immersion} and \eqref{submersion} are analogous. For the sake of completeness, we will describe here the proof for the submersion property \eqref{submersion}.

First, given a morphism $\phi$, such that $\phi|_U:U\to V$ is defined by \eqref{submdef}, one directly verifies that $d\phi_m$ is of the form
$$
d\phi_m\simeq \left(\begin{array}{cc|cc|cc|cc}
								\mathbb{I}& 0 & & & & & &\\
								\hline
								& &\mathbb{I}&0  & & & &\\
								\hline
								& & & & \ddots & & &\\
								\hline
								& & & & & &\mathbb{I}& 0
						\end{array}\right)
						\in \gl^0(r|\mathbf{s}\times p|\mathbf{q},\R)\;,
$$
which is of rank $r|\mathbf{s}$. Thus, $\phi$ is a submersion at $m$.

For the converse, let us consider $\Zn$-charts $(V,(y,\theta_\mu))$ of $N$ around $|\phi|(m)$ and $(U,(x,\xi_{\mu}))$ of $M$ around $m$ (with $|U|\subset|\phi|^{-1}(|V|)$). In the following we consider $\phi|_U:U\to V$. Up to reordering of the coordinates, we can assume that the sub-blocks
\begin{align} \label{hyp1}
 &\left(\partial_{x^j} \phi^*(y^i)\Big|_{m}\right)_{i,j=1, \ldots,r} \mbox{ of the diagonal block $B_0$ } \\
\mbox{and }\label{hyp2} &\left(\partial_{\xi_{\mu}^b} \phi^*(\theta_{\mu}^a)\Big|_{m}\right)_{a,b=1, \ldots,s_{\mu}} \mbox{ of the diagonal blocks $B_{\mu}$}\;,
\end{align}
are invertible matrices.
Let us consider the $\Zn$-superdomain $\R^{p-r|\mathbf{q}-\mathbf{s}}$ and relabel its coordinates as
$(y^{r+i},$ $\theta_{\mu}^{s_{\mu}+a_{\mu}})$, $1\leq i \leq p-r$ and $1\leq a_{\mu}\leq q_{\mu}-s_{\mu}$. Thanks to the fundamental theorem of $\Zn$-morphisms \cite[Theorem 6.8]{Covolo:2016}, we can construct a morphism $\psi:U\to V\times \R^{p-r|\mathbf{q}-\mathbf{s}}$ by setting
$$
\psi^*(y^i)= \begin{cases}
								\phi^*(y^i) & \mbox{ if } 1\leq i\leq r\\
								x^i & \mbox{ if } r+1\leq i\leq p
						\end{cases}\;,
$$
and, for all $\mu\in\Zn\setminus\{0\}$,
$$
\psi^*(\theta_{\mu}^{a})= \begin{cases}
																\phi^*(\theta_{\mu}^{a}) & \mbox{ if } 1\leq a\leq s_{\mu}\\
																\xi_{\mu}^{a} & \mbox{ if } s_{\mu}+1\leq a\leq q_{\mu}
															\end{cases}
\;.$$
Note that, by construction, we have $\phi|_U =\pi_1^{V\times\R^{p-r|\mathbf{q}-\mathbf{s}}}\circ \psi$.
By hypothesis \eqref{hyp1}, \eqref{hyp2}, $d\psi_{m}$ is invertible and so by Theorem \ref{InvFunThm} $\psi$ is a local diffeomorphism. Hence, there are subdomains $U'\subset U$ around $m$, $V'\subset V$ around $|\phi|(m)$, and $W'\subset\R^{p-r}$, such that $\psi:U'\to V'\times W'$ is a diffeomorphism. Considering the new coordinates on $U'$ given by $\widetilde{x}^i:=\psi^*(y^i)$ and $\widetilde{\xi}_{\mu}^a:=\psi^*(\theta_{\mu}^a)$, the claim \eqref{submersion} or \eqref{submdef} follows.
\end{proof}

%%%%%%%%%%%%%%%%%%%%%%%%%%%%%%%%%%%%%%%%%%%%%%%%%%%%%%%%%%%%%%%%%%%%%%%%%%%%%%%%%%%%%%%%%%%%%%%%%%%%%%%%%%%%%%%%%%%%%%%%%%%%%%%%%%%

%------------------------------------------------
\subsection{Constant rank theorem}
%-----------------------------------------------

Let $\phi:M\to M'$ be a $\Z_2^n$-morphism of $\Zn$-manifolds of dimension $p|\vect{q}$ and $p'|\vect{q'}$, respectively. Let $m\in|M|$, let $V$ be a $\Z_2^n$-chart of $M'$ around $|\phi|(m)\in|M'|$, and let $U$ be a $\Z_2^n$-chart of $M$ around $m$, such that $|U|\subset |\phi|^{-1}(|V|)\subset |M|$. We can restrict $\phi$ to a $\Z_2^n$-morphism $\phi:U\to V$. Let $r \in \N$ and $\vect{s}=(s_{\mu})_{\mu \in \Zn\setminus\{0\}}$, with $r\leq \min(p,p')$ and $s_{\mu}\leq \min(q_{\mu},q'_{\mu})$, for all $\mu\in \Zn\setminus\{0\}$.

\begin{defn}\label{Def:cstrank}
The (graded) Jacobian matrix $\op{Jac}_{\phi}\in \gl^0(p'|\vect{q'} \times p|\vect{q}\,,\cO(U))$ is said to be of {\it constant rank} $r|{\bf s}$, if there exist $G_1 \in \GL(p'|{\bf q'}, \cO(U))$ and $G_2 \in \GL(p|{\bf q}\,,\cO(U))$, such that

\be\label{cstrank}
G_1 \op{Jac}_{\phi} G_2 =
\left(\begin{array}{c|c|c|c}
\begin{array}{cc}
\mathbb{I}_r & 0\\
0 & 0
\end{array}
									&    &  &    \\
\hline
 &  \begin{array}{cc}
		\mathbb{I}_{s_1} & 0\\
			0 & 0
		\end{array}  						&    &  \\
\hline
 &  & \ddots  &  \\
\hline
  &   &  & \begin{array}{cc}
\mathbb{I}_{s_N} & 0\\
0 & 0
\end{array} \\
\end{array}\right).
\ee
\end{defn}

Notice that we defined the concept of constant rank for the Jacobian matrix of a $\Z_2^n$-morphism $\phi$ written in $\Z_2^n$-charts $(U,\psi_1)$ and $(V,\psi'_1)$ around $m$ and $|\phi|(m)$, respectively. The coordinate $\Z_2^n$-diffeomorphism $\psi_1:U\to \mathcal{U}$ (resp., $\psi'_1:V\to\mathcal{V}$) allowed us to identify the open $\Z_2^n$-submanifold $U$ of $M$ (resp., $V$ of $M'$) with the $\Z_2^n$-domain $\mathcal{U}$ (resp., $\mathcal{V}$), so that we got $\psi_1\simeq \id$ (resp., $\psi'_1\simeq \id$). In other words, we actually defined the notion of constant rank for the Jacobian matrix of $$\phi_1:=\psi'_1\phi\,\psi_1^{-1}\;.$$

\begin{rema}\label{CstRankIndCoord}
If we choose new coordinate diffeomorphisms $\psi_2:U\to \mathcal{U}$ and $\psi_2':V\to\mathcal{V}$ (we could also consider $\psi'_1:V_1\to\mathcal{V}_1$ and $\psi'_2:V_2\to \mathcal{V}_2$, and similarly for $\psi_1$ and $\psi_2$), the Jacobian matrix of $\phi_2:=\psi'_2\phi\,\psi_2^{-1}$ has constant rank $r|\mathbf{s}$ if and only if this holds for the Jacobian matrix of $\phi_1$.
\end{rema}

Indeed, assume for instance that Equation \eqref{cstrank} is satisfied for $\op{Jac}_{\phi_2}$. Since $\psi:=\psi_1\psi_2^{-1}:\mathcal{U}\to\mathcal{U}$ (resp., $\psi'^{-1}:=\psi_2'\psi_1'^{-1}:\mathcal{V}\to\mathcal{V}$) is a $\Z_2^n$-diffeomorphism, its tangent map is invertible in $\mathcal{U}$ (resp., in $\mathcal{V}$), i.e., $$\varepsilon(\op{Jac}_{\psi})\quad (\text{resp.,}\quad\varepsilon(\op{Jac}_{\psi'^{-1}}))$$ is invertible over $\Ci(\mathcal{U})$ (resp., over $\Ci(\mathcal{V})$). The proof of Proposition \ref{invertibleblockdiag}, see Appendix, implies that $$\mathcal{G}_2:=\op{Jac}_{\psi}\in\GL(p|\mathbf{q},\mathcal{O}(\mathcal{U}))\quad (\text{resp.,}\quad\mathcal{G}_1:=\op{Jac}_{\psi'^{-1}}\in\GL(p'|\mathbf{q}',\mathcal{O}(\mathcal{V})))\;.$$ The existence of $\mathcal{G}_1^{-1}$ such that $\mathcal{G}_1^{-1}\mathcal{G}_1=\mathcal{G}_1\mathcal{G}_1^{-1}=\mathbb{I}$, implies that $$(\phi_1\psi)^*(\mathcal{G}_1)\in\GL(p'|\mathbf{q}',\mathcal{O}(\mathcal{U}))\;.$$ Using $\psi'^{-1}\phi_1\psi=\phi_2$, we get $$(\psi^{-1})^*\big(G_1\,(\phi_1\psi)^*(\mathcal{G}_1)\big)\;\op{Jac}_{\phi_1}\;\,(\psi^{-1})^*\big(\mathcal{G}_2\,G_2\big)=$$ $$(\psi^{-1})^*\big(G_1\,(\phi_1\psi)^*(\mathcal{G}_1)\;\psi^*(\op{Jac}_{\phi_1})\;\,\mathcal{G}_2\,G_2\big)=$$ $$(\psi^{-1})^*\big(G_1\op{Jac}_{\phi_2}G_2\big)\;,$$ so that $\op{Jac}_{\phi_1}$ has constant rank $r|\mathbf{s}$.

\begin{rema}\label{CstRankDifferential}
When evaluating Equality \eqref{cstrank} at a point $m$, this equality remains valid (see Appendix, Proposition \ref{invertibleblockdiag}). Hence, if $\op{Jac}_{\phi}$ is of \emph{constant rank} $r|\vect{s}$ (in the sense of Definition \ref{Def:cstrank}), then $d\phi_m\simeq\op{Jac}_{\phi}|_m\in\gl^0(p'|\mathbf{q}'\times p|\mathbf{q},\R)$ is a block-diagonal matrix of \emph{rank} $r|\mathbf{s}$ (in the sense of Subsection \ref{ImmSub}). As in the super case, the converse is not true.
\end{rema}

\begin{thm}[Constant rank theorem for $\Zn$-supermanifolds]\label{thm:cstrk}
Let $\phi:M\to M'$ be a morphism of $\Zn$-manifolds, of respective dimensions $p|\mathbf{q}$ and $p'|\mathbf{q'}$, and let $m$ be a point in $|M|$. Then the following are equivalent:
\begin{enumerate}
	\item In a neighborhood of $m$, the Jacobian matrix $\op{Jac}_{\phi}$ is a graded matrix of constant rank $r|\mathbf{s}$.
	\item In a neighborhood $U$ of $m$, $\phi$ may be written as the composite of a submersion $\phi_1:U\to W$ at $m$ and an immersion $\phi_2:W\to V$ at $|\phi_1|(m)$.
\end{enumerate}
\end{thm}

\begin{rema}\label{CstRankCoordExpr}
In view of Proposition \ref{prop:sub-im}, (2) implies that there exist coordinate charts $(U,(x,\xi_{\mu}))$ near $m$ and $(V,(y,\eta_{\mu}))$ near $|\phi|(m)$, in which the morphism $\phi$ writes as
\beas
			\phi^*(y^i)=\begin{cases}
														x^i\,, &\mbox{ if } 1\leq i\leq r \\
														0\,, &\mbox{ if } r+1\leq i\leq p'
									\end{cases}
			& \mbox{ and } &
			\phi^*(\eta_{\mu}^{a})=\begin{cases}
																					\xi_{\mu}^{a}\,, &\mbox{ if } 1\leq a\leq s_{\mu} \\
																					0\,, &\mbox{ if } s_{\mu}+1\leq a\leq q'_{\mu}
														  \end{cases}\;.
\eeas
\end{rema}

\medskip
\begin{proof}[Proof of Theorem \ref{thm:cstrk}]
That (2) implies (1) clearly follows from Remark \ref{CstRankCoordExpr}. The converse is more involved: first one needs to construct a submersion $\phi_1$ and an immersion $\phi_2$, and then prove equality between their composite $\Phi:=\phi_2\circ \phi_1$ and the original morphism $\phi$.

Since the statement is local, we may restrict ourselves as before to consider a morphism $\phi:U\to V$ of $\Zn$-domains with coordinates $$u=(x,\xi_{\mu})\quad\text{and}\quad v=(y,\eta_{\mu})\;.$$ We assume that $m\in|U|$ has coordinates $m\simeq x=0$ (centered $\Z_2^n$-chart). In the following, we often distinguish between the first $(r,s_\mu)$ coordinates or components, and the remaining ones. We then write $$(u',u'')=(x',x'',\xi'_\mu,\xi''_\mu)\quad\text{and}\quad (v',v'')=(y',y'',\eta'_\mu,\eta''_\mu)\;.$$ We use similar notation for base morphisms, for instance for $|\phi|:|U|\subset\R^p\to |V|\subset\R^{p'}$, and write $$|\phi|=(|\phi|',|\phi|'')\;,$$ where $|\phi|'$ (resp., $|\phi|''$) are the $r$ first (the $p'-r$ remaining) $\R$-valued component-functions of $|\phi|$.

Let us assume that $\op{Jac}_{\phi}$ has constant rank $r|\mathbf{s}$ near the point $m$. Thus, the block-diagonal matrix $d\phi_m=\op{Jac}_{\phi}|_m$ is of rank $r|\mathbf{s}$, see Remark \ref{CstRankDifferential}. Up to reordering of the coordinates, we may assume that the subblocks of $d\phi_m$ that are invertible are
\bea\label{diff}
\partial_{x'} \phi^*(y')|_m
& \mbox{ and }&
\partial_{\xi'_\mu} \phi^*(\eta'_\mu)|_m\;,
\eea
for all $\mu\in \Zn\setminus\{0\}$.

We denote by $(z,\zeta_{\mu})$ the coordinates of the $\Zn$-domain $\R^{r|\mathbf{s}}$ and denote by $W$ the open $\Z_2^n$-subdomain of $\R^{r|\mathbf{s}}$ with base $|W|=\{y'\,|\;\exists y'':(y',y'')\in |V|\}$. Let $$\phi_1:U\to W$$ be the morphism of $\Zn$-domains defined by
\bea\label{oldform}
\phi_1^*(z)= \phi^*(y')\in \big(\cO^0(|U|)\big)^{\times r}& \mbox{ and } & \phi_1^*(\zeta_{\mu})= \phi^*(\eta'_{\mu})\in\big(\cO^\mu(|U|)\big)^{\times s_\mu}\;,
\eea
for all $\mu\in\Zn\setminus\{0\}$. We have $$|\phi_1|=\varepsilon(\phi_1^*(z))=\varepsilon(\phi^*(y'))=|\phi|':|U|\to |W|\;.$$ By \eqref{diff}, the differential of $\phi_1$ at $m$ has rank $r|\mathbf{s}$, and so $\phi_1$ is a submersion at this point. By the proof of the second item of Proposition \ref{prop:sub-im}, shrinking the domain $U$ near $m$ if necessary, we can define new coordinates $(\widetilde{x},\widetilde{\xi}_{\mu})$ (changing the first $r|\mathbf{s}$ coordinates, i.e., $\widetilde{x}''=x''$ and $\widetilde{\xi}_\mu''=\xi_\mu''$) such that the morphism $\phi_1$  writes as
\bea\label{newform}
\phi_1^*(z)=\widetilde{x}' &\mbox{  and  } & \phi_1^*(\zeta_{\mu})=\widetilde{\xi}_{\mu}'\;.
\eea

Let us now consider the morphism of $\Zn$-domains $$\psi:W\to U$$ defined by
\bea\label{PsiWU}
\begin{cases}
							\psi^*(\widetilde{x}') = z\\
							\psi^*(\widetilde{x}'') = 0
						\end{cases}
\mbox{ and } \;\;\;\;
\begin{cases}
                            \psi^*(\widetilde{\xi}_{\mu}')= \zeta_{\mu}\\
							\psi^*(\widetilde{\xi}_{\mu}'')= 0
						\end{cases}.
\eea
Let us denote by $U'$ the $\Z_2^n$-subdomain of $U$ defined by the equations
$$
\widetilde{x}''=0\quad\text{and}\quad\widetilde{\xi}_\mu''=0\;\, (\text{considering all nonzero degrees}\;\mu)\;,
$$
see Remark \ref{rem:subdomain}. Notice that the coordinates $(\widetilde{x}',\widetilde{x}'',\widetilde{\xi}'_\mu,\widetilde{\xi}''_\mu)$ of $m$ are $(\widetilde{x}',0,0,0)$, so that, with some minor abuse of notation, we have $m\in |U'|$. The morphisms $\phi_1$ and $\psi$ actually implement a $\Z_2^n$-diffeomorphism between $U'$ and $W$. Indeed, the $\Z_2^n$-morphism $\widetilde\phi_1:U'\to W$, defined by $$
\widetilde\phi_1^*(z):=\phi_1^*(z)=\widetilde{x}'\quad\text{and}\quad\widetilde\phi_1^*(\zeta_\mu):=\phi_1^*(\zeta_\mu)=\widetilde{\xi}_\mu'\;,
$$
and the $\Z_2^n$-morphism $\widetilde{\psi}:W\to U'$, defined by
\be\label{psi*identif}
\widetilde{\psi}^*(\widetilde{x}'):=\psi^*(\widetilde{x}')=z\quad\text{and}\quad\widetilde{\psi}^*(\widetilde{\xi}_\mu'):=\psi^*(\widetilde{\xi}_\mu')=\zeta_\mu\;,
\ee
are inverses. In the following, we use the identification
\be\label{IdentU'W}
U'\simeq W\;.
\ee

The composite $\Z_2^n$-morphism $$\phi_2:=\phi\circ \psi: W\to V$$ pulls the first coordinates $(y',\eta'_\mu)$ back to
\be\label{newform2}
\phi_2^*(y')= z\quad\text{and}\quad\phi_2^*(\eta_{\mu}') = \zeta_{\mu}\;.
\ee
One checks, by direct computation of the differential of $\phi_2$ at $|\phi_1|(m)$, that $\phi_2$ is an immersion at this point. By the first item of Proposition \ref{prop:sub-im}, shrinking $V$ near $|\phi_2|(|\phi_1|(m))$ if necessary, we can find new coordinates $(\widetilde{y},\widetilde{\eta}_{\mu})$ (changing the last $p'-r|\mathbf{q'}-\mathbf{s}$ coordinates, i.e., $\widetilde{y}'=y'$ and $\widetilde{\eta}'_\mu=\eta'_\mu$) such that the morphism $\phi_2$ writes as
\be\label{newform3}
\phi_2^*(\widetilde{y}')= z\,,\;\phi_2^*(\widetilde{y}'')=0\,,\;\phi_2^*(\widetilde{\eta}_{\mu}')=\zeta_\mu\,,\;\text{ and }\;\phi_2^*(\widetilde{\eta}_{\mu}'')=0\;.
\ee

It remains now to show equality between the original morphism $\phi$ and the composite $\Phi:=\phi_2\circ \phi_1$ in a neighborhood of $m$. By the fundamental theorem of $\Zn$-morphisms \cite{Covolo:2016}[Theorem 6.8], it suffices to prove the equality around $m$ of the $\Zn$-functions $\phi^*(\widetilde v)$ and $\Phi^*(\widetilde v)$, for all coordinates $\widetilde v$ on $V$.

On the one hand, it follows from Equations \eqref{newform3} and \eqref{oldform} that
$$
\Phi^*(\widetilde y')=\phi_1^*(\phi_2^*(\widetilde y'))= \phi^*(\widetilde y')\;.
$$
On the other hand, since
$$
\phi^*(\widetilde y'')=\sum_{\alpha,\beta}\; f_{\alpha,\beta}(\widetilde{x}',\widetilde{x}'')\;\widetilde{\xi}'^\alpha\widetilde{\xi}''^\beta\;,
$$
we get, using \eqref{newform3}, \eqref{IdentU'W}, \eqref{psi*identif}, and \eqref{PsiWU},
$$
\Phi^*(\widetilde y'')|_{U'} = \phi_1^*(\phi_2^*(\widetilde{y}''))|_{U'} = 0 = \phi_2^*(\widetilde{y}'') = \psi^*(\phi^*(\widetilde{y}'')) = \sum_{\alpha}\, f_{\alpha,0}(\widetilde x',0)\;\widetilde \xi'^{\alpha}=\phi^*(\widetilde y'')|_{U'}\;.
$$
The same argument holds for the nonzero degree coordinates $\widetilde\eta'_{\nu}, \widetilde\eta''_{\nu}$:
\bea \label{=onU'}
\Phi^*(\widetilde{y})|_{U'}=\phi^*(\widetilde{y})|_{U'} & \mbox{ and } &  \Phi^*(\widetilde{\eta}_{\nu})|_{U'}=\phi^*(\widetilde{\eta}_{\nu})|_{U'}\;.
\eea
We still need to ``extend'' these equalities to $U$. This extension will be based on Lemma \ref{lem:eqf}. It thus suffices to prove that the conditions \eqref{diff0} are satisfied for each pair $(\Phi^*(\widetilde{y}^i), \phi^*(\widetilde{y}^i))$ and for each pair $(\Phi^*(\widetilde{\eta}_\nu^a), \phi^*(\widetilde{\eta}_\nu^a))$.

We first compute the derivatives of $\Phi^*(\widetilde{y}), \Phi^*(\widetilde{\eta}_\nu)\in\cO(U)$ with respect to $\widetilde{x}''$ and $\widetilde{\xi}_\mu''\,$. Setting $\widetilde{u}''=(\widetilde{x}'',\widetilde{\xi}_\mu'')$ and $w=(z,\zeta)$, and recalling that $\widetilde{v}=(\widetilde{y},\widetilde{\eta}_\nu)$, we obtain, with a slight abuse of notation,

$$
(D):=\partial_{\widetilde{u}''}\Phi^*(\widetilde{v})=\partial_{\widetilde{u}''}\phi_1^*(\phi_2^*(\widetilde{v}))=\sum_k\; \partial_{\widetilde{u}''}\phi_1^*(w^k)\; \phi_1^*\big(\partial_{w^k}\phi_2^*(\widetilde{v})\big)=:\sum_k\; \partial_{\widetilde{u}''}(1)\; \phi_1^*\big(\partial_{w^k} (2)\big)\;.
$$
If $\widetilde{v}=\widetilde{y}''$ or $\widetilde{v}=\widetilde{\eta}''$, then $(2)=0$, so $(D)=0$. If $\widetilde{v}=\widetilde{y}'$ (resp., $\widetilde{v}=\widetilde{\eta}'$), then $(2)=z$ (resp., $(2)=\zeta$), so $(D)=$ $\partial_{\widetilde{u}''}\widetilde{x}'=0$ (resp., $(D)=\partial_{\widetilde{u}''}\widetilde{\xi}'=0$). Hence, all the derivatives of interest vanish for $\Phi^*$.

As for $\phi^*$, let us recall that all pullbacks of target coordinates are formal power series in the source nonzero degree coordinates with smooth coefficients with respect to the source zero degree coordinates. In view of Equations \eqref{oldform} and \eqref{newform}, we have in particular
$$
\phi^*(\widetilde{y}')=\widetilde{x}'\quad\text{and}\quad\phi^*(\widetilde{\eta}_\mu')=\widetilde{\xi}'_\mu\;.
$$
The Jacobian matrix of $\phi$ thus reads
\be\label{JacFirst}
\quad\quad\quad\quad\;\op{Jac}_\phi=
\left(
\begin{array}{cc|cc|cc|cc}
\mathbb{I}_r & 0 & 0 & 0 & \ldots & &0 & 0 \\
\star_{11} & \bullet_{11} & \star_{12} & \bullet_{12} & \ldots & &\star_{12^n} &\bullet_{12^n} \\
\hline
0 & 0 &\mathbb{I}_{s_1} & 0 &\ldots & &0 & 0 \\
\star_{21} & \bullet_{21} & \star_{22} & \bullet_{22} & \ldots & &\star_{22^n} &\bullet_{22^n} \\
\hline
\vdots & \vdots & \vdots & \vdots &\ddots & & \vdots& \vdots \\
 %&  &  &  &  & & & \\
\hline
0 & 0 &0 & 0 &\ldots &  &\mathbb{I}_{s_{N}} & 0\\
\star_{2^n1} & \bullet_{2^n1} & \star_{2^n2} & \bullet_{2^n2} & \ldots & &\star_{2^n2^n} &\bullet_{2^n2^n} \\
\end{array}
\right)
\in \gl^0(p'|\mathbf{q'} \times p|\mathbf{q};\cO(U))\;\;.
\ee

\noindent By assumption $\op{Jac}_\phi$ is of constant rank $r|\vect{s}$ near $m$ -- whatever the coordinates considered, see Remark \ref{CstRankIndCoord}. Hence, there exist invertible degree-zero matrices $G_1$,$G_2$ of $\Zn$-functions, such that $G_1\op{Jac}_\phi G_2$ is of the form \eqref{cstrank}, i.e., of the form \eqref{JacFirst} with all $\star_{ij}$ and $\bullet_{ij}$ equal to zero. Recall that the matrices $G_1^{-1}$ and $G_2$ have the standard block-decomposition. Moreover, in view of the decomposition of coordinates into $'$ and $''$ coordinates, we can decompose each block-row (resp., block-column) into a $'$ and a $''$ row (resp., column). For instance, the $'\,''$ subblocks are the subblocks $B_{ij}$.
\be\label{StructureG1G2}
\left(
\begin{array}{cc|cc|cc|cc}
A_{11} & B_{11} & A_{12} & B_{12} & \ldots & &A_{12^n} & B_{12^n} \\
C_{11} & D_{11} & C_{12} & D_{12} & \ldots & &C_{12^n} &D_{12^n} \\
\hline
A_{21} & B_{21} &A_{22} & B_{22} &\ldots & &A_{22^n} & B_{22^n} \\
C_{21} & D_{21} & C_{22} & D_{22} & \ldots & &C_{22^n} &D_{22^n} \\
\hline
\vdots & \vdots & \vdots & \vdots &\ddots & & \vdots& \vdots \\
 %&  &  &  &  & & & \\
\hline
A_{2^n1} & B_{2^n1} &A_{2^n2} & B_{2^n2} &\ldots &  &A_{2^n2^n} & B_{2^n2^n}\\
C_{2^n1} & D_{2^n1} & C_{2^n2} & D_{2^n2} &\ldots  & &C_{2^n2^n} &D_{2^n2^n} \\
\end{array}
\right)\;.
\ee
In particular,
$$
\op{Jac}_\phi G_2=G_1^{-1}(G_1\op{Jac}_\phi G_2)
$$
has all its $''$ columns equal to $0$. To see this, compute $G_1^{-1}(G_1\op{Jac}_\phi G_2)$, i.e., apply $G_1^{-1}=\eqref{StructureG1G2}$ to $\eqref{JacFirst}$ with $\star=\bullet=0$, and compute more precisely any $''$ column. Now compute $\op{Jac}_\phi G_2$, i.e., multiply $\eqref{JacFirst}$ and $G_2=\eqref{StructureG1G2}$, and compute more precisely any $'\,''$ subblock -- which we know to be zero. It follows that in $G_2$ all the subblocks $B_{ij}$ vanish, so that we get from Proposition \ref{invertibleblockdiag} that all subblocks $D_{ii}$ are invertible. Compute finally in $\op{Jac}_\phi G_2$ any $''\,''$ subblock $k\ell$ -- which we know to be zero. The result is $\sum_{i=1}^{2^n} \bullet_{ki}\; D_{i\ell}=0$, for any $k,\ell$, i.e.,
$$
(\bullet_{k1},\bullet_{k2},\ldots,\bullet_{k2^n})\; D:=
(\bullet_{k1},\bullet_{k2},\ldots,\bullet_{k2^n})
\left(
\begin{array}{c|c|c|c}
D_{11} & D_{12} & \ldots &D_{12^n} \\
\hline
D_{21} & D_{22} & \ldots &D_{22^n} \\
\hline
\vdots & \vdots &  \ddots & \vdots \\
 %&  &  &  &  & & & \\
\hline
D_{2^n1} & D_{2^n2} &  \ldots &D_{2^n2^n} \\
\end{array}
\right)
=0\;,$$
for any $k$. Since $D$ is invertible, again due to Proposition \ref{invertibleblockdiag}, it follows that $\bullet_{k\ell}=0$, for all $k,\ell$, so that the $''$ columns of $\op{Jac}_\phi$ in \eqref{JacFirst} do all vanish. This means exactly that, for any $\widetilde{v}$ and any $\widetilde{u}''$, the derivative $\partial_{\widetilde{u}''}\phi^*(\widetilde v)=0$ and thus coincides with the corresponding derivative $\partial_{\widetilde{u}''}\Phi^*(\widetilde{v})$.

It follows that all the pullbacks $\phi^*(\widetilde{v})$ and $\Phi^*(\widetilde{v})$ coincide on some neighborhood of $m$ (see Lemma \ref{lem:eqf} below), so that $\phi=\phi_2\circ\phi_1$ on this neighborhood.
\end{proof}

\begin{lem}\label{lem:eqf}
Let $U$ be a $\Zn$-domain with $r+k|\mathbf{s}+\mathbf{l}$ coordinates $$u=(x,\xi)=(x',x'',\xi',\xi'')\;.$$ Consider the $\Zn$-subdomain $U'$ of $U$ defined by the equations
\bea \label{U'def}
x''=0  \;\mbox{ and  }\; \xi''=0\;.
\eea
If two $\Z_2^n$-functions $f_1,f_2$ on $U$ are such that
\be\label{eqonU'}
f_1|_{U'}=f_2|_{U'}
\ee
and
\be
\partial_{x''}f_1=\partial_{x''}f_2\quad\text{and}\quad \partial_{\xi''}f_1=\partial_{\xi''}f_2\label{diff0}\;,
\ee
then $f_1$ and $f_2$ coincide on some neighbourhood of $U'$.
\end{lem}

\begin{rema}\label{rem:subdomain}
Let us be more precise regarding the $\Zn$-subdomain $U'$ of $U$ defined by \eqref{U'def}. Its base manifold is $$|U'|=\{x':(x',0'')\in|U|\}\subset\R^r\;,$$ whereas its structure sheaf is $$\cO_{|U'|}(-)=\Ci_{|U'|}(-)[[\xi']]\;\;.$$ There is a natural embedding $\rho: U'\to U$ given by
$$
\begin{cases}\rho^*(u') = u' \\ \rho^*(u'') = 0\end{cases}.
$$
Then, for any $f\in \cO(U)=\cO_{|U|}(|U|)$, we set $f|_{U'}:=\rho^*(f)\in\cO(U')=\cO_{|U'|}(|U'|)$. If $f=f(x',x'',\xi',\xi'')$, then $$f|_{U'}=f(x',0,\xi',0)\;.$$
\end{rema}

\begin{proof}[Proof of Lemma \ref{lem:eqf}]
The $\Z_2^n$-function $f=f_1-f_2\in \cO(U)$ reads
$$
f=f(x',x'',\xi',\xi'')=\sum_{\alpha, \beta} f_{\alpha,\beta}(x',x'')\xi'^{\alpha}\xi''^{\beta}\quad (f_{\alpha,\beta}\in\Ci(|U|))\;.
$$
By $\cJ(U)$-adic continuity of derivations, any partial derivative of the power series $f$ is given by deriving term by term. When taking into account the second equation of \eqref{diff0}, we get, for any $b\in\{1,\ldots,|\mathbf{l}|\}$,
$$
0=\partial_{\xi''^b}f=\sum_{\alpha,\beta}\pm\beta_b f_{\alpha,\beta}(x',x'') \xi'^{\alpha} \xi''^{\beta-e_b}\;,
$$
where $e_b$ is vector number $b$ of the canonical basis of $\R^{|\mathbf{l}|}$. This implies that
\be\label{ab1}
f_{\alpha,\beta} =0\,,\quad \mbox{for all } \alpha \mbox{ and for all } \beta\neq 0 \;.
\ee
When taking into account the first equation of \eqref{diff0}, we obtain
$$
0=\partial_{x''}f=\sum_{\alpha,\beta}\;\partial_{x''}\!f_{\alpha,\beta}\;\xi'^\alpha \xi''^\beta\;,
$$
i.e., we see that, in some neighborhood $|U'|_\varepsilon:=\{(x',x'')\in|U|:x'\in|U'|, x''\simeq 0\}$ of $|U'|$,
\be\label{ab2}
f_{\alpha,\beta}(x',x'')= f_{\alpha,\beta}(x',0),\quad\text{for all}\; \alpha,\beta\;.
\ee
Combining $\eqref{ab1}$ and \eqref{ab2}, we finally get
$$
f=\sum_\alpha f_{\alpha,0}(x',x'')\xi'^\alpha = \sum_\alpha f_{\alpha,0}(x',0)\xi'^\alpha = f(x',0,\xi',0) = f|_{U'}=0\;,
$$
in $|U'|_{\varepsilon}$, in view of \eqref{eqonU'}.
\end{proof}

%-%-%-%-%-%-%-%-%-%-%-%-%-%-%
\section{Appendix}
%-%-%-%-%-%-%-%-%-%-%-%-%-%-%

%------------------------------------------------
\subsection{Linear algebra over Hausdorff-complete $\Zn$-commutative rings.}
%------------------------------------------------

In this section, we will work exclusively with $\Zn$-commutative rings which are Hausdorff-complete in the $J$-adic topology, where $J$ denotes the (proper) homogeneous ideal of $R$ generated by the elements of nonzero degree $\gamma_i\in\Z_2^n$, $i\in\{1,\ldots,N\}$, $N=2^n-1$.

\subsubsection{Rank of a linear map.}

We begin with the following criterion for invertibility of a square degree zero matrix.

\begin{prop}\label{invertibleblockdiag}
Let $R$ be a $J$-adically Hausdorff-complete $\Zn$-commutative ring, and let $T\in\gl^0(p|\mathbf{q},R)$ be a degree zero square $p|{\bf q}$ matrix with entries in $R$, written in the standard block format

\begin{equation*}
T =  \left(\begin{array}{c|c|c}
T_{00} & \dotsc & T_{0N} \\
\hline
\vdots & \ddots & \vdots \\
\hline
T_{N0} & \dotsc & T_{NN} \\
\end{array}\right).
\end{equation*}

\noindent The matrix $T$ is invertible if and only if  $T_{ii}$ is invertible for all $i$.
\end{prop}

In this work, we are of course mainly interested in the case $R=\cO(U)$ and $J=\cJ(U)$.

\newcommand{\ve}{\varepsilon}
\newcommand{\tve}{\tilde\varepsilon}

\begin{proof}

The $\Z_2^n$-graded ring morphism $\varepsilon: R \to \,R/J$ induces a map $$\tilde\varepsilon:\gl^0(p|\mathbf{q},R)\ni T\mapsto \tilde\varepsilon(T)\in\op{Diag}(p|\mathbf{q},R/J),$$ where $\tilde\varepsilon(T)$ is the block-diagonal matrix with diagonal blocks $\tilde\varepsilon(T_{ii})$ (with commuting entries). Notice that $\tilde\varepsilon(T)$ is invertible if and only if all $\tilde\varepsilon(T_{ii})$ are invertible. In this case, the inverse $(\tve(T))^{-1}$ is the block-diagonal matrix with diagonal blocks $(\tve(T_{ii}))^{-1}$.

Clearly, if $T$ is invertible over $R$, i.e., if there is a matrix $T^{-1}$ such that $TT^{-1}=T^{-1}T=\mathbb{I}$, then $\tilde\varepsilon(T)$ is invertible over $R/J$. More precisely, since $\varepsilon$ is a ring morphism, we have $$\tilde\varepsilon(T)\tilde\varepsilon(T^{-1})=\tilde\varepsilon(T^{-1})\tilde\varepsilon(T)=\tilde\varepsilon(\mathbb{I})\;,$$ so that $\tilde\varepsilon(T^{-1})=(\tilde\varepsilon(T))^{-1}$.

Conversely, assume that $\tilde\varepsilon(T)$ is invertible over $R/J$. Its inverse is a block-diagonal matrix with diagonal blocks made of elements in $R/J$, i.e., elements of the type $\ve(r)$ with $r$ in degree zero. Hence, there is a (block-diagonal) zero-degree matrix $Y$ over $R$, such that $\tve(Y)\tve(T)=\tve(T)\tve(Y)=\tve(\mathbb{I})$, or, equivalently, such that $YT=\mathbb{I}+Z_L$ and $TY=\mathbb{I}+Z_R$, with $Z_L,Z_R\in \mathrm{gl}^0(p|\mathbf{q},J)$. By Hausdorff-completeness, the matrix $\mathbb{I}+Z$, for $Z=Z_L$ or $Z=Z_R$, is invertible with inverse
$$ (\mathbb{I}+Z)^{-1}= \mathbb{I}+\sum_{k\geq 1} (-Z)^k\in\gl^0(p|\mathbf{q},R)\;. $$
It follows that $T_L^{-1}T:=(\mathbb{I}+Z_L)^{-1}YT=\mathbb{I}$, that $TT_R^{-1}:=TY(\mathbb{I}+Z_R)^{-1}=\mathbb{I}$, and that $T_L^{-1}=T_L^{-1}TT_R^{-1}=T_R^{-1}$, so that $T$ is invertible over $R$, with inverse $T_L^{-1}=T_R^{-1}$.

Finally $T$ is invertible if and only if $\tve(T)$ is invertible, if and only if $\tve(T_{ii})$ is invertible for all $i$, if and only if $T_{ii}$ is invertible for all $i$.
\end{proof}

\vspace{1cm}

\noindent Tiffany Covolo

\noindent National Research University Higher School of Economics, Moscow

\noindent covolotiffany@gmail.com\bigskip

\noindent Stephen Kwok

\noindent University of Luxembourg

\noindent sdkwok2@gmail.com\bigskip

\noindent Norbert Poncin

\noindent University of Luxembourg

\noindent norbert.poncin@uni.lu

\end{document}